\documentclass[11pt]{article}
\usepackage[tbtags]{amsmath}
\usepackage{amssymb}
\usepackage{amsthm}
\usepackage[misc]{ifsym}
\usepackage{cases}
\usepackage{mathrsfs}
\usepackage{graphicx}
\usepackage{float}
\usepackage[colorlinks,linkcolor=blue,anchorcolor=blue,citecolor=blue]{hyperref}
\usepackage{relsize}
\usepackage{caption}
\usepackage{subfigure}

\numberwithin{equation}{section}   
\normalsize
\title{\bf Mixed Leadership Stochastic Differential Game in Feedback Information Pattern with Applications \thanks{This work is supported by National Key R$\&$D Program of China (2022YFA1006100), National Natural Science Foundations of China (11971266, 11831010), and Shandong Provincial Natural Science Foundations (ZR2022JQ01, ZR2020ZD24, ZR2019ZD42).}}
\author{\normalsize  Qi Huang\thanks{\it School of Mathematics, Shandong University, Jinan 250100, P.R. China, E-mail: 201911814@mail.sdu.edu.cn} , Jingtao Shi\thanks{Corresponding author, \it School of Mathematics, Shandong University, Jinan 250100, P.R. China, E-mail: shijingtao@sdu.edu.cn}}

\newtheorem{mythm}{Theorem}[section]
\newtheorem{mydef}{Definition}[section]

\newtheorem{Remark}{Remark}[section]
\begin{document}
\maketitle

\noindent{\bf Abstract:}\quad This paper is devoted to a high-dimensional mixed leadership stochastic differential game on a finite horizon in feedback information mode, where the control variables enter into the diffusion term of state equation. A verification theorem for the feedback Stackelberg-Nash equilibrium is obtained by using a system of coupled and fully nonlinear parabolic partial differential equations. We apply the verification theorem to deal with a dynamic innovation and pricing decision problem where the buyer acts as the leader in the pricing decisions and the dynamic model is stochastic. Via the solutions of coupled Riccati equations, we explicitly express the feedback equilibrium strategies of innovation and pricing. And by analysis, the local existence and uniqueness of the solutions of the coupled Riccati equations is derived. We also conduct some numerical analyses to discuss the effects of model parameters on the feedback equilibrium strategies.

\vspace{2mm}


\vspace{2mm}

\section{Introduction}

\hspace{0.4cm} In 1934, von Stackelberg first proposed the concept of Stackelberg solution in \cite{St1934}. In the context of Stackelberg two-person game, players have asymmetric roles, with one leading and the other following. Firstly, the follower solves his optimization problem for any decision announced by the leader. The follower's optimal response, which is a function of the leader's decision, will be obtained. Considering the follower's optimal response, the leader solves his optimal problem and gets his optimal policy under this response. By substituting the leader's optimal decision into the follower's response, the follower's optimal decision can be derived. These two decisions together make up the Stackelberg solution.\\
\indent There are a variety of equilibrium concepts in dynamic games due to different types of dominance of the leader over the follower and various information sets. Refer to Ba\c{s}ar and Olsder \cite{BO1998} for the concept of different equilibria in discrete-time and continuous-time deterministic cases. If the leader has a global advantage over the follower and has the ability to announce his policy for the entire duration at the beginning of the game, then the corresponding solution is called the global Stackelberg solution. And open-loop Stackelberg solution is one of the global Stackelberg solution. However, the global Stackelberg solution is time inconsistent, and both players must commit to their respective strategies which are predetermined at the beginning. See Bensoussan et al. \cite{BSP2015} for a detailed description of the way in which the global Stackelberg solution is played.\\
\indent If the leader moves ahead of the follower in every period based on the observed state, the corresponding solution is called the feedback Stackelberg solution which is time consistent. Therefore, the leader only has a phased advantage over the follower in this solution concept. Simaan and Cruz \cite{SC1973-1, SC1973-2} first introduced the feedback Stackelberg solution. Ba\c{s}ar and Haurie \cite{BH1984} argued that the continuous-time game problem can be regarded as the limit of discrete-time game sequences, but the convergence of discrete-time game solutions is a challenging open problem. Inspired by the results of the feedback Nash equilibrium, they circumvented the limit problem by defining a feedback Stackelberg equilibrium concept and derived the corresponding system of coupled {\it Hamilton-Jacobi-Bellman} (HJB, in short) equations to characterize the equilibrium. Bensoussan et al. \cite{BCS2014} studied an infinite-horizon Stackelberg stochastic differential game which involved Brownian motion, and obtained a sufficient condition.\\
\indent Most of the literature on Stackelberg games assumes that the characters of the players are fixed from the beginning with the leader remaining the leader and the follower remaining the follower throughout the game. Ba\c{s}ar \cite{B1973} first raised the question of whether being a leader is always beneficial to the player and proved that leadership was not always the first choice for the player. In fact, in many practical situations, there are no clear leaders or followers. For example, in a supply chain where a manufacturer sells a product through a retailer, the leader in advertising strategies is the manufacturer, while the leader in pricing strategies is the retailer. In this case, the manufacturer acts as a follower to make the wholesale price decision, and the retailer acts as a follower to make the decision related to local advertising of the product. In 2010, Ba\c{s}ar et al. \cite{BAS2010} first introduced the concept of the mixed leadership game, in which the same player can be a leader in some decisions and be a follower in others. And by applying the maximum principle, the necessary condition for the optimal equilibrium in open-loop information mode is obtained. The open-loop Stackelberg equilibrium of the {\it linear-quadratic} (LQ, in short) mixed leadership differential game was studied by Bensoussan et al. in \cite{BCS2013}. And under some coefficient assumptions, they proved that there existed a unique Stackelberg solution. Xie et al. \cite{XFH2021} considered a mixed LQ Stackelberg stochastic differential game with input constraint in open-loop information pattern. Different from the above literature which all discussed the open-loop solution of the mixed leadership game, Bensoussan et al. \cite{BCCS2019} gave the definition of the feedback Stackelberg-Nash equilibrium for the mixed leadership stochastic differential game under the feedback information mode. At the same time, they obtained a sufficient condition for the equilibrium, which was used to deal with a manufacturer-retailer cooperative advertising game problem. However, the problem studied in Bensoussan et al. \cite{BCCS2019} was one-dimensional and the diffusion term of the state equation didn't include two players' control variables.\\
\indent Inspired by the above literatures, in this paper we consider a high-dimensional mixed leadership stochastic differential game with a finite time zone in the feedback information pattern. In contrast to \cite{BCCS2019}, all the control variables of the two players enter the drift term and diffusion term of the state equation. And we derive a verification theorem for the feedback Stackelberg-Nash equilibrium by using the system of coupled HJB equations. Since they studied a one-dimensional problem with no controls in the diffusion term of the state equation, the second derivative terms of the value functions with respect to the state variable didn't enter into the Hamiltonian functions and the corresponding system of HJB equations in the verification theorem is the system of {\it ordinary differential equations} (ODEs, in short). Different from \cite{BCCS2019}, the Hamiltonian functions we consider contain the second partial derivative terms of the value functions with respect to the state variable due to the entry of controls into the diffusion term. In addition, the problem that we deal with is high dimension, so our system of coupled HJB equations is the one of coupled and fully nonlinear parabolic {\it partial differential equations} (PDEs, in short).\\
\indent We apply the above verification theorem to solve a dynamic innovation and pricing decision problem in a supply chain. There is a seller who sells the product to a buyer, and then the buyer sells it to the final consumers. What they need to do is making their own pricing strategies and innovation strategies. Song et al. \cite{SCDL2021} modeled the dynamic innovation and pricing strategy problem as a Stackelberg game where the seller as the leader first announces the wholesale price and his innovation effort, and the buyer as the follower decides on the retail price and his corresponding innovation effort to react. They derived the feedback equilibrium strategies and got useful managerial insights through analytical and numerical methods. See Giovanni \cite{G2011}, El Ouardighi et al. \cite{OJP2008} for more information on the dynamic game of innovation and pricing strategy issues. We note that the state equation in the game model they consider is a deterministic ODE. However, in practice, the dynamic innovation and pricing strategies may be disturbed by some random factors. Therefore, inspired by Tapiero \cite{T1975-1, T1975-2}, we extend the dynamic model of \cite{SCDL2021} to the stochastic case. In \cite{SCDL2021}, the seller acts as the leader in both innovation and pricing decisions. Whereas, according to Ba\c{s}ar et al. \cite{BAS2010}, Mozafari et al. \cite{MSS2021}, and Taleizadeh et al. \cite{TCCM2021}, we know that the large buyer like Walmart acts as the leader in pricing strategies in the actual supply chain. Thus, the seller is only the leader in innovation strategies. Given the leadership policies, the seller as a follower gives the wholesale price and the buyer as a follower gives his innovation effort. In this case, there are no clear leaders or followers. Therefore, the Stackelberg game is no longer applicable, and we need to apply the mixed leadership game theory to deal with this problem.\\
\indent By using the verification theorem which we obtain in the mixed leadership stochastic differential game, we solve the above supply chain problem in which the seller and the buyer are the leaders of innovation strategies and pricing strategies respectively. And via the solutions of coupled Riccati equations, the corresponding explicit representation of the feedback Stackelberg-Nash equilibrium is derived. According to Picard-Lindel{\"o}f's Theorem (\cite{AO2008}), we obtain the local existence and uniqueness of the solutions of the coupled Riccati equations. Finally, we conduct some numerical analyses to investigate the impact of representative model parameters $C_0$ and $\delta(t)$ on players' feedback equilibrium strategies and get some managerial insights.\\
\indent This paper is organized as follows. In the feedback information pattern, section \ref{section 2.1} formulates a high-dimensional mixed leadership stochastic differential game with players' control variables in the drift and the diffusion terms of the state equation. In section \ref{section 2.2}, we use the system of coupled HJB equations to derive a verification theorem for the feedback Stackelberg-Nash equilibrium. Based on this result, we study a dynamic innovation and pricing strategy problem in the supply chain in section \ref{section 3} and apply the solutions of coupled Riccati equations to explicitly express the feedback equilibrium strategies. Moreover, the local existence and uniqueness of the solutions of the coupled Riccati equations is guaranteed by Picard-Lindel{\"o}f's Theorem. In section \ref{section 4}, we numerically analyze the effects of representative model parameters on the feedback equilibrium strategies, and obtain some managerial insights.

\section{Mixed leadership stochastic differential game}\label{section 2}

\subsection{Problem formulation}\label{section 2.1}

\hspace{0.4cm} In this paper, $(\Omega,\mathcal{F},\{\mathcal{F}_t\}_{t\ge 0},\mathbb{P})$ is a complete probability space, and $W(\cdot)$ is a $m$-dimensional standard Brownian motion defined on it. For any $T>0$ and initial state $x_0 \in \mathbb R^n$, we consider a stochastic differential game with two players, 1 and 2. The state equation of the game is
\begin{equation}
\left\{
             \begin{aligned}
             dx(s)&=f\big(s,x(s),u_1(s),v_1(s),u_2(s),v_2(s)\big)ds\\
                  &\quad +\sigma\big(s,x(s),u_1(s),v_1(s),u_2(s),v_2(s)\big)dW(s),\quad 0 \le s \le T, \\
              x(0)&=x_0.
             \end{aligned}
\right.
\end{equation}
Here $f:[0,T]\times \mathbb{R}^n \times \mathbb{U}_1 \times \mathbb{V}_1 \times \mathbb{U}_2 \times \mathbb{V}_2 \to \mathbb{R}^n$ and $\sigma:[0,T]\times \mathbb{R}^n \times \mathbb{U}_1 \times \mathbb{V}_1 \times \mathbb{U}_2 \times \mathbb{V}_2 \to \mathbb{R}^{n\times m}$ are the drift and diffusion terms of the state equation, respectively. $(u_i(\cdot), v_i(\cdot))\in \mathbb{U}_i \times \mathbb{V}_i \subseteq \mathbb{R}^{k_i}\times \mathbb{R}^{z_i}$ are the decision variables of the player $i, i=1,2.$\\
\indent The expected profit functionals for the player $i$ to maximize is as follows:
\begin{equation}\label{cost functional-player i}
J_i(u_1,v_1,u_2,v_2)=\mathbb E\left[\int^T_0 g_i\big(s,x(s),u_1(s),v_1(s),u_2(s),v_2(s)\big)ds+h_i(x(T))\right],
\end{equation}
for $i=1,2$. Here $g_i:[0,T]\times \mathbb{R}^n \times \mathbb{U}_1 \times \mathbb{V}_1 \times \mathbb{U}_2 \times \mathbb{V}_2 \to \mathbb{R}$ and $h_i:\mathbb{R}^n \to \mathbb{R}(i=1,2)$ satisfy the following assumption.

{\bf (H1)}\quad $g_i,h_i,i=1,2$ are measurable functions which satisfy polynomial growth conditions, namely there exist positive constants $C$ and $p\ge 2$ such that
\begin{equation*}
\begin{aligned}
|g_i(s,x,u_1,v_1,u_2,v_2)|&\le C(1+|x|^p+|u_1|^p+|v_1|^p+|u_2|^p+|v_2|^p),\\
|h_i(x)|&\le C(1+|x|^p),\\
\forall (s,x,u_1,v_1,u_2,v_2)&\in [0,T]\times \mathbb{R}^n \times \mathbb{U}_1 \times \mathbb{V}_1 \times \mathbb{U}_2 \times \mathbb{V}_2.
\end{aligned}
\end{equation*}
And $g_i$ is continuously differentiable in $(u_1,v_1,u_2,v_2)\in \mathbb{U}_1 \times \mathbb{V}_1 \times \mathbb{U}_2 \times \mathbb{V}_2, i=1,2.$\\ \indent In the mixed leadership stochastic differential game under feedback information pattern, both players act as leaders in $u$ strategies and followers in $v$ strategies. At each instant of time, according to the observed state, player 1 and player 2 respectively and simultaneously determine their instantaneous leadership strategies $u_1(s,x(s))$ and $u_2(s,x(s))$ for any $s\in [0,T]$. Next, based on the observed state and the instantaneous leadership strategies, player 1 and player 2 make their respective follow-up instantaneous responses $v_1(s,x(s),u_1(s,x(s)),u_2(s,x(s)))$ and $v_2(s,x(s),u_1(s,x(s)),u_2(s,x(s)))$ by playing a feedback Nash game. Knowing the instantaneous Nash responses at the follow-up layer, player 1 and player 2 play a feedback Nash game again and get the corresponding feedback Nash equilibria in the leadership layer. Substitute the leadership feedback strategies into the feedback Nash responses of the follow-up layer and obtain the follow-up feedback strategies. There is a Stackelberg game vertically nested in two Nash games for the reason that the follow-up strategies are affected by the leadership strategies. Therefore, the admissible strategy spaces for the two players are defined as follows. For $i=1,2$,
\begin{equation*}
\begin{aligned}
&\mathcal{U}_i=\Big\{u_i \big|u_i:[0,T]\times {\mathbb{R}}^n \to \mathbb{U}_i, u_i(s,x) \textrm{ is uniformly locally Lipschitz continuous in } x,\textrm{ and}\\
&\qquad\qquad \textrm{ there exist constants } a, b>0 \textrm{ such that } |u_i(s,x)|\le a+b|x|,\ \forall (s,x) \in [0,T]\times {\mathbb{R}}^n \Big\},\\
&\mathcal{V}_i=\Big\{v_i \big|v_i:[0,T]\times {\mathbb{R}}^n \times \mathbb{U}_1 \times \mathbb{U}_2 \to \mathbb{V}_i, v_i(s,x,\mu_1,\mu_2) \textrm{ is uniformly locally Lipschitz }\\
&\qquad\qquad \textrm{continuous in } (x,\mu_1,\mu_2) \textrm{ and there exist constants } a, b>0 \textrm{ such that }\\
&\qquad\qquad |v_i(s,x,\mu_1,\mu_2)|\le a+b(|x|+|\mu_1|+|\mu_2|),\ \forall (s,x,\mu_1,\mu_2) \in [0,T]\times {\mathbb{R}}^n \times \mathbb{U}_1 \times \mathbb{U}_2\Big\}.
\end{aligned}
\end{equation*}
\indent For any $(t,x)\in [0,T]\times {\mathbb{R}}^n$ and $(u_1,v_1,u_2,v_2)\in \mathcal{U}_1 \times \mathcal{V}_1 \times \mathcal{U}_2 \times \mathcal{V}_2$, consider the following parameterized state equation:
\begin{equation}\label{the parameterized state equation}
\left\{
             \begin{aligned}
             dx(s)&=f\big(s,x(s),u_1(s,x(s)),v_1(s,x(s),u_1(s,x(s)),u_2(s,x(s))),\\
                  &\qquad u_2(s,x(s)),v_2(s,x(s),u_1(s,x(s)),u_2(s,x(s)))\big)ds\\
                  &\quad +\sigma(s,x(s),u_1(s,x(s)),v_1(s,x(s),u_1(s,x(s)),u_2(s,x(s))),\\
                  &\qquad u_2(s,x(s)),v_2(s,x(s),u_1(s,x(s)),u_2(s,x(s))))dW(s),\  t \le s \le T, \\
              x(t)&=x.
             \end{aligned}
\right.
\end{equation}
\indent To ensure the existence of the unique solution to $(\ref{the parameterized state equation})$, we assume that $f$ and $\sigma$ satisfy the following conditions:\\
{\bf (H2)}\quad (i) $f:[0,T]\times \mathbb{R}^n \times \mathbb{U}_1 \times \mathbb{V}_1 \times \mathbb{U}_2 \times \mathbb{V}_2 \to \mathbb{R}^n$ and $\sigma:[0,T]\times \mathbb{R}^n \times \mathbb{U}_1 \times \mathbb{V}_1 \times \mathbb{U}_2 \times \mathbb{V}_2 \to \mathbb{R}^{n\times m}$ are measurable functions. $f$ and $\sigma$ are uniformly locally Lipschitz continuous in $(x,u_1,v_1,u_2,v_2)$.\\
(ii) $f$ and $\sigma$ are continuously differentiable in $(u_1,v_1,u_2,v_2)\in \mathbb{U}_1 \times \mathbb{V}_1 \times \mathbb{U}_2 \times \mathbb{V}_2$.\\
(iii) There exist constants $c, d>0$ such that
\begin{equation*}
\begin{aligned}
&|f(s,x,u_1,v_1,u_2,v_2)|\le c+d(|x|+|u_1|+|v_1|+|u_2|+|v_2|),\\
&|\sigma(s,x,u_1,v_1,u_2,v_2)|\le c+d(|x|+|u_1|+|v_1|+|u_2|+|v_2|),\\
&\qquad \forall (s,x,u_1,v_1,u_2,v_2)\in [0,T]\times \mathbb{R}^n \times \mathbb{U}_1 \times \mathbb{V}_1 \times \mathbb{U}_2 \times \mathbb{V}_2.
\end{aligned}
\end{equation*}
\indent When the strategies $(u_1,v_1,u_2,v_2)$ selected from the context is unambiguous, we denote by $x^{t,x}(\cdot)$ of the solution to the parameterized state equation $(\ref{the parameterized state equation})$. The corresponding expected profit functionals of the player $i$ is
\begin{equation}\label{parameterized cost functionals}
\begin{split}
&J^{t,x}_i\big(u_1(\cdot,\cdot),v_1(\cdot,\cdot,u_1(\cdot,\cdot),u_2(\cdot,\cdot)),u_2(\cdot,\cdot),v_2(\cdot,\cdot,u_1(\cdot,\cdot),u_2(\cdot,\cdot))\big)\\
&=\mathbb E\bigg[\int^T_t g_i\big(s,x^{t,x}(s),u_1(s,x^{t,x}(s)),v_1(s,x^{t,x}(s),u_1(s,x^{t,x}(s)),u_2(s,x^{t,x}(s))),\\
&\qquad\qquad\qquad u_2(s,x^{t,x}(s)),v_2(s,x^{t,x}(s),u_1(s,x^{t,x}(s)),u_2(s,x^{t,x}(s)))\big)ds+h_i(x^{t,x}(T))\bigg].
\end{split}
\end{equation}

\begin{mydef}\label{Definition}
For any $(t,x)\in [0,T]\times \mathbb{R}^n$, if the following holds:
\begin{equation}\label{feedback Stackelberg-Nash equilibrium}
\left\{
\begin{aligned}
&J^{t,x}_1 (u^*_1(\cdot,\cdot),v^*_1(\cdot,\cdot,u^*_1(\cdot,\cdot),u^*_2(\cdot,\cdot)),u^*_2(\cdot,\cdot),v^*_2(\cdot,\cdot,u^*_1(\cdot,\cdot),u^*_2(\cdot,\cdot)))\\
&\ge J^{t,x}_1 (u_1(\cdot,\cdot),v_1(\cdot,\cdot,u_1(\cdot,\cdot),u^*_2(\cdot,\cdot)),u^*_2(\cdot,\cdot),v^*_2(\cdot,\cdot,u_1(\cdot,\cdot),u^*_2(\cdot,\cdot))),\\
&\hspace{5cm}\forall (u_1,v_1)\in \mathcal{U}_1 \times \mathcal{V}_1,\\
&J^{t,x}_2 (u^*_1(\cdot,\cdot),v^*_1(\cdot,\cdot,u^*_1(\cdot,\cdot),u^*_2(\cdot,\cdot)),u^*_2(\cdot,\cdot),v^*_2(\cdot,\cdot,u^*_1(\cdot,\cdot),u^*_2(\cdot,\cdot)))\\
&\ge J^{t,x}_2 (u^*_1(\cdot,\cdot),v^*_1(\cdot,\cdot,u^*_1(\cdot,\cdot),u_2(\cdot,\cdot)),u_2(\cdot,\cdot),v_2(\cdot,\cdot,u^*_1(\cdot,\cdot),u_2(\cdot,\cdot))),\\
&\hspace{5cm}\forall (u_2,v_2)\in \mathcal{U}_2 \times \mathcal{V}_2,\\
\end{aligned}
\right.
\end{equation}
the two pairs of strategies $(u^*_i,v^*_i)\in \mathcal{U}_i \times \mathcal{V}_i, i=1,2$, are called the feedback Stackelberg-Nash equilibrium.
\end{mydef}

\subsection{Verification theorem}\label{section 2.2}

\hspace{0.4cm} We will provide a verification theorem which is a sufficient condition for the feedback Stackelberg-Nash equilibrium in this subsection.

Let $\mathcal{S}^n$ represent the set of symmetric $n\times n$ matrices $A=(A_{ij}), i,j=1, \cdots, n.$ Set $a=\sigma {\sigma}^\top$ and $\mathrm{tr}[aA]=\sum\limits_{i=1}^n \sum\limits_{j=1}^n a_{ij}A_{ij}$. Let $C_p([0,T]\times \mathbb{R}^n)$ represent the set of all continuous functions $\Psi(s,x)$ on $[0,T]\times \mathbb{R}^n$ which satisfy the following polynomial growth condition, namely there exists a positive constant $C$ and $p$ mentioned above such that
\[|\Psi(s,x)| \le C(1+|x|^p).\]
Let $C^{1,2}([0,T]\times \mathbb{R}^n)$ represent the set of all continuous functions $\Psi(s,x)$ on $[0,T]\times \mathbb{R}^n$ with continuous partial derivative in $s$ and 2-order continuous partial derivative in $x$.

We introduce the Hamiltonian functions for the two players as follows:
\begin{equation*}
\begin{aligned}
&H_1(s,x,\mu_1,\nu_1,\mu_2,\nu_2,p_1,A')\triangleq\big\langle p_1,f(s,x,\mu_1,\nu_1,\mu_2,\nu_2)\big\rangle+\frac{1}{2}\mathrm{tr}\big[a(s,x,\mu_1,\nu_1,\mu_2,\nu_2)A'\big]\\
&\qquad\qquad\qquad\qquad\qquad\qquad\qquad +g_1(s,x,\mu_1,\nu_1,\mu_2,\nu_2),
\end{aligned}
\end{equation*}
\begin{equation}
\begin{aligned}
&H_2(s,x,\mu_1,\nu_1,\mu_2,\nu_2,p_2,A'')\triangleq\big\langle p_2,f(s,x,\mu_1,\nu_1,\mu_2,\nu_2)\big\rangle+\frac{1}{2}\mathrm{tr}\big[a(s,x,\mu_1,\nu_1,\mu_2,\nu_2)A''\big]\\
&\qquad\qquad\qquad\qquad\qquad\qquad\qquad +g_2(s,x,\mu_1,\nu_1,\mu_2,\nu_2),
\end{aligned}
\end{equation}
where $H_i:[0,T]\times \mathbb{R}^n \times \mathbb{U}_1 \times \mathbb{V}_1 \times \mathbb{U}_2 \times \mathbb{V}_2 \times \mathbb{R}^n \times \mathcal{S}^n \to \mathbb{R}$, $i=1,2$.\\
\indent Given a pair of leadership actions $(\mu_1,\mu_2)\in \mathbb{U}_1 \times \mathbb{U}_2,$ we first solve a static Nash game with $\nu_1$ and $\nu_2$ as the strategy variables at the level of Hamiltonian. Suppose there exists a unique solution $(\Gamma_1^{\nu}(s,x,\mu_1,\mu_2,p_1,p_2,A',A''),\Gamma_2^{\nu}(s,x,\mu_1,\mu_2,p_1,p_2,A',A''))$ to the following equations
\begin{equation}
\left\{
\begin{aligned}
&\Gamma_1^{\nu}\big(s,x,\mu_1,\mu_2,p_1,p_2,A',A''\big)\\
&\triangleq\mathop{argmax}\limits_{\nu_1 \in\, \mathbb{V}_1} H_1\big(s,x,\mu_1,\nu_1,\mu_2,\Gamma_2^{\nu}(s,x,\mu_1,\mu_2,p_1,p_2,A',A''),p_1,A'\big),\\
&\Gamma_2^{\nu}\big(s,x,\mu_1,\mu_2,p_1,p_2,A',A''\big)\\
&\triangleq\mathop{argmax}\limits_{\nu_2 \in\, \mathbb{V}_2} H_2\big(s,x,\mu_1,\Gamma_1^{\nu}(s,x,\mu_1,\mu_2,p_1,p_2,A',A''),\mu_2,\nu_2,p_2,A''\big).
\end{aligned}
\right.
\end{equation}
\indent Considering the response at the $\nu$ level, we substitute $(\Gamma_1^{\nu},\Gamma_2^{\nu})$ into $H_1, H_2$ and solve the corresponding static Nash game between $\mu_1$ and $\mu_2$ at the level of Hamiltonian. We assume that the following equations have a unique solution $\big(\Gamma_1^{\mu}(s,x,p_1,p_2,A',A''),\Gamma_2^{\mu}(s,x,p_1,p_2,A',A'')\big)$:
\begin{equation}
\left\{
\begin{aligned}
&\Gamma_1^{\mu}(s,x,p_1,p_2,A',A'')\\
&\triangleq\mathop{argmax}\limits_{\mu_1 \in\, \mathbb{U}_1} H_1\big(s,x,\mu_1,\Gamma_1^{\nu}(s,x,\mu_1,\Gamma_2^{\mu}(s,x,p_1,p_2,A',A''),p_1,p_2,A',A''),\\
&\qquad\quad \Gamma_2^{\mu}(s,x,p_1,p_2,A',A''),\Gamma_2^{\nu}(s,x,\mu_1,\Gamma_2^{\mu}(s,x,p_1,p_2,A',A''),p_1,p_2,A',A''),p_1,A'),\\
&\Gamma_2^{\mu}(s,x,p_1,p_2,A',A''\big)\\
&\triangleq\mathop{argmax}\limits_{\mu_2 \in\, \mathbb{U}_2} H_2\big(s,x,\Gamma_1^{\mu}(s,x,p_1,p_2,A',A''),\\
&\qquad\quad \Gamma_1^{\nu}(s,x,\Gamma_1^{\mu}(s,x,p_1,p_2,A',A''),\mu_2,p_1,p_2,A',A''),\\
&\qquad\quad \mu_2,\Gamma_2^{\nu}(s,x,\Gamma_1^{\mu}(s,x,p_1,p_2,A',A''),\mu_2,p_1,p_2,A',A''),p_2,A''\big).
\end{aligned}
\right.
\end{equation}

\indent We give the following verification theorem.

\begin{mythm}
Let $(H1)$ and $(H2)$ hold. Suppose $V_1(s,x)$, $V_2(s,x)$ both lie in $C^{1,2}([0,T]\times \mathbb{R}^n) \cap C_p([0,T]\times \mathbb{R}^n)$ and there exist $(u^*_1,v^*_1)\in \mathcal{U}_1 \times \mathcal{V}_1, (u^*_2,v^*_2)\in \mathcal{U}_2 \times \mathcal{V}_2,$ where
\begin{equation}\label{proof equilibrium definition 1}
\left\{
\begin{aligned}
&u_1^*(s,x) \triangleq \Gamma_1^{\mu}\Big(s,x,\frac{\partial V_1}{\partial x},\frac{\partial V_2}{\partial x},\frac{{\partial}^2 V_1}{\partial x^2},\frac{{\partial}^2 V_2}{\partial x^2}\Big),\\
&u_2^*(s,x) \triangleq \Gamma_2^{\mu}\Big(s,x,\frac{\partial V_1}{\partial x},\frac{\partial V_2}{\partial x},\frac{{\partial}^2 V_1}{\partial x^2},\frac{{\partial}^2 V_2}{\partial x^2}\Big),\\
\end{aligned}
\right.
\end{equation}
\begin{equation}\label{proof equilibrium definition 2}
\left\{
\begin{aligned}
&v_1^*(s,x,u_1^*(s,x),u_2^*(s,x)) \triangleq \Gamma_1^{\nu}\Big(s,x,u_1^*(s,x),u_2^*(s,x),\frac{\partial V_1}{\partial x},\frac{\partial V_2}{\partial x},\frac{{\partial}^2 V_1}{\partial x^2},\frac{{\partial}^2 V_2}{\partial x^2}\Big),\\
&v_2^*(s,x,u_1^*(s,x),u_2^*(s,x)) \triangleq \Gamma_2^{\nu}\Big(s,x,u_1^*(s,x),u_2^*(s,x),\frac{\partial V_1}{\partial x},\frac{\partial V_2}{\partial x},\frac{{\partial}^2 V_1}{\partial x^2},\frac{{\partial}^2 V_2}{\partial x^2}\Big),\\
\end{aligned}
\right.
\end{equation}
such that
\begin{equation}\label{HJB-1}
\left\{
\begin{aligned}
&\frac{\partial V_1}{\partial s}(s,x)+H_1\Big(s,x,u_1^*(s,x),v_1^*(s,x,u_1^*(s,x),u_2^*(s,x)),\\
&\qquad\qquad\qquad\quad u_2^*(s,x),v_2^*(s,x,u_1^*(s,x),u_2^*(s,x)),\frac{\partial V_1}{\partial x},\frac{{\partial}^2 V_1}{\partial x^2}\Big)=0,\\
&V_1(T,x)=h_1(x),
\end{aligned}
\right.
\end{equation}
\begin{equation}\label{HJB-2}
\left\{
\begin{aligned}
&\frac{\partial V_2}{\partial s}(s,x)+H_2\Big(s,x,u_1^*(s,x),v_1^*(s,x,u_1^*(s,x),u_2^*(s,x)),\\
&\qquad\qquad\qquad\quad u_2^*(s,x),v_2^*(s,x,u_1^*(s,x),u_2^*(s,x)),\frac{\partial V_2}{\partial x},\frac{{\partial}^2 V_2}{\partial x^2}\Big)=0,\\
&V_2(T,x)=h_2(x),
\end{aligned}
\right.
\end{equation}
for all $(s,x)\in [0,T]\times \mathbb{R}^n$. Then $\big\{(u_1^*(\cdot,\cdot),v_1^*(\cdot,\cdot,u_1^*(\cdot,\cdot),u_2^*(\cdot,\cdot))),(u_2^*(\cdot,\cdot),v_2^*(\cdot,\cdot,u_1^*(\cdot,\cdot),u_2^*(\cdot,\cdot)))\big\}$
is the feedback Stackelberg-Nash equilibrium.

\end{mythm}

\begin{proof}
Define
\begin{equation}\label{proof equilibrium definition 3}
\begin{aligned}
&v_1^*(s,x,\mu_1,\mu_2)\triangleq \Gamma_1^{\nu}\Big(s,x,\mu_1,\mu_2,\frac{\partial V_1}{\partial x},\frac{\partial V_2}{\partial x},\frac{{\partial}^2 V_1}{\partial x^2},\frac{{\partial}^2 V_2}{\partial x^2}\Big),\\
&v_2^*(s,x,\mu_1,\mu_2)\triangleq \Gamma_2^{\nu}\Big(s,x,\mu_1,\mu_2,\frac{\partial V_1}{\partial x},\frac{\partial V_2}{\partial x},\frac{{\partial}^2 V_1}{\partial x^2},\frac{{\partial}^2 V_2}{\partial x^2}\Big),
\end{aligned}
\end{equation}
for any $(s,x)\in [0,T]\times \mathbb{R}^n, (\mu_1, \mu_2)\in \mathbb{U}_1 \times \mathbb{U}_2.$\\
\indent Suppose player 1 adopts any pair of strategies $(u_1, v_1)\in \mathcal{U}_1 \times \mathcal{V}_1$ and player 2 adopts the pair of strategies $(u_2^*(\cdot,\cdot),v_2^*(\cdot,\cdot,u_1(\cdot,\cdot),u_2^*(\cdot,\cdot))),$ where $u_2^*(\cdot,\cdot)$ is defined by $(\ref{proof equilibrium definition 1})$ and $v_2^*(\cdot,\cdot,u_1(\cdot,\cdot),u_2^*(\cdot,\cdot))$ is defined by $(\ref{proof equilibrium definition 3}).$ For any $(t,x)\in [0,T]\times \mathbb{R}^n,$ we denote by $x^{t,x}(\cdot)$ of the solution to the corresponding parameterized state equation.\\
\indent For sake of notation, we will omit the arguments of the policies $u_i$ and $v_i$ in $f, \sigma, g_i, i=1,2.$ Let $B_r(x)$ denote an open ball of radius $r$, centered at $x$, i.e.,
\[B_r(x)\triangleq \left\{y\in \mathbb{R}^n \arrowvert \sqrt{\sum\limits_{i=1}^n |x_i-y_i|^2}<r\right\}.\]
And let $\tau_r$ represent the first exit time of $x^{t,x}(\cdot)$ from $B_r(x)$, i.e.,
\[\tau_r \triangleq \inf\big\{s|x^{t,x}(s)\notin B_r(x),\ t\le s\le T\big\}.\]
If $x^{t,x}(s)\in B_r(x)$ for all $s\in [t,T],$ we set $\tau_r=T$.\\
\indent By applying It\^{o}'s formula to $V_1(s,x^{t,x}(s))$, integrating from $t$ to $\tau_r \wedge T$, and taking expectation, we obtain
\begin{equation*}
\begin{split}
&V_1(t,x)=\mathbb E\big[V_1(\tau_r \wedge T,x^{t,x}(\tau_r \wedge T))\big]\\
&\qquad\qquad -\mathbb E\Bigg\{\int_t^{\tau_r \wedge T}\bigg[\frac{\partial V_1}{\partial s}(s,x^{t,x}(s))+\bigg\langle\frac{\partial V_1}{\partial x}(s,x^{t,x}(s)),f\big(s,x^{t,x}(s),u_1,v_1,u_2^*,v_2^*\big)\bigg\rangle\\
&\qquad\qquad\qquad +\frac{1}{2} \mathrm{tr}\Big[a\big(s,x^{t,x}(s),u_1,v_1,u_2^*,v_2^*\big)\frac{{\partial}^2 V_1}{\partial x^2}(s,x^{t,x}(s))\Big]\\
&\qquad\qquad\qquad +g_1\big(s,x^{t,x}(s),u_1,v_1,u_2^*,v_2^*\big)\bigg] ds \Bigg\}+\mathbb E\int_t^{\tau_r \wedge T}g_1\big(s,x^{t,x}(s),u_1,v_1,u_2^*,v_2^*\big)ds\\\
&\qquad\quad =\mathbb E\Big[\int_t^{\tau_r \wedge T}g_1\big(s,x^{t,x}(s),u_1,v_1,u_2^*,v_2^*\big)ds +V_1(\tau_r \wedge T,x^{t,x}(\tau_r \wedge T))\Big]\\
&\qquad\qquad -\mathbb E\Bigg\{\int_t^{\tau_r \wedge T}\bigg[\frac{\partial V_1}{\partial s}(s,x^{t,x}(s))+H_1\Big(s,x^{t,x}(s),u_1,v_1,u_2^*,v_2^*,\\
&\qquad\qquad\qquad\qquad \frac{\partial V_1}{\partial x}(s,x^{t,x}(s)),\frac{{\partial}^2 V_1}{\partial x^2}(s,x^{t,x}(s))\Big) \bigg]ds \Bigg\}\\
&\qquad\quad \ge \mathbb E\Big[\int_t^{\tau_r \wedge T}g_1\big(s,x^{t,x}(s),u_1,v_1,u_2^*,v_2^*\big)ds +V_1(\tau_r \wedge T,x^{t,x}(\tau_r \wedge T))\Big]\\
\end{split}
\end{equation*}
\begin{equation*}
\begin{split}
&\qquad\qquad -\mathbb E\Bigg\{\int_t^{\tau_r \wedge T}\bigg[\frac{\partial V_1}{\partial s}(s,x^{t,x}(s))+\sup\limits_{\mu_1 \in\, \mathbb{U}_1}\Big[\sup\limits_{\nu_1 \in\, \mathbb{V}_1} H_1\Big(s,x^{t,x}(s),\mu_1,\nu_1,u_2^*(s,x^{t,x}(s)),\\
&\qquad\qquad\qquad\qquad v_2^*(s,x^{t,x}(s),\mu_1,u_2^*(s,x^{t,x}(s))),\frac{\partial V_1}{\partial x}(s,x^{t,x}(s)),\frac{{\partial}^2 V_1}{\partial x^2}(s,x^{t,x}(s))\Big) \Big] \bigg] ds \Bigg\}\\
&\qquad\quad =\mathbb E\Big[\int_t^{\tau_r \wedge T}g_1\big(s,x^{t,x}(s),u_1,v_1,u_2^*,v_2^*\big)ds +V_1(\tau_r \wedge T,x^{t,x}(\tau_r \wedge T))\Big].
\end{split}
\end{equation*}
We can apply the dominated convergence theorem because of the polynomial growth condition on $g_i$ and $V_i, i=1,2.$ By sending $r \to \infty$, we get
\begin{equation}\label{player-1-proof-1}
\begin{split}
V_1(t,x)&\ge \mathbb E\Big[\int_t^{T}g_1\big(s,x^{t,x}(s),u_1,v_1,u_2^*,v_2^*\big)ds +h_1(x^{t,x}(T))\Big]\\
&= J^{t,x}_1 \big(u_1(\cdot,\cdot),v_1(\cdot,\cdot,u_1(\cdot,\cdot),u^*_2(\cdot,\cdot)),u^*_2(\cdot,\cdot),v^*_2(\cdot,\cdot,u_1(\cdot,\cdot),u^*_2(\cdot,\cdot))\big),
\end{split}
\end{equation}
for any $(u_1,v_1)\in \mathcal{U}_1 \times \mathcal{V}_1.$\\
\indent When player 1 adopts the pair of strategies $(u_1^*(\cdot,\cdot),v_1^*(\cdot,\cdot,u_1^*(\cdot,\cdot),u_2^*(\cdot,\cdot)))$ and player 2 adopts the pair of strategies $(u_2^*(\cdot,\cdot),v_2^*(\cdot,\cdot,u_1^*(\cdot,\cdot),u_2^*(\cdot,\cdot)))$, we use $x^{t,x,*}(\cdot)$ to represent the solution to the corresponding parameterized state equation. Analogously, if we apply It\^{o}'s formula to $V_1(s,x^{t,x,*}(s))$ and follow the procedure above, we can obtain
\begin{equation}\label{player-1-proof-2}
\begin{split}
V_1(t,x)&= \mathbb
E\Big[\int_t^{T}g_1\big(s,x^{t,x,*}(s),u_1^*,v_1^*,u_2^*,v_2^*\big)ds +h_1(x^{t,x,*}(T))\Big]\\
&= J^{t,x}_1 \big(u_1^*(\cdot,\cdot),v_1^*(\cdot,\cdot,u_1^*(\cdot,\cdot),u^*_2(\cdot,\cdot)),u^*_2(\cdot,\cdot),v^*_2(\cdot,\cdot,u_1^*(\cdot,\cdot),u^*_2(\cdot,\cdot))\big).
\end{split}
\end{equation}
According to $(\ref{player-1-proof-1})$ and $(\ref{player-1-proof-2})$, we derive
\begin{equation*}
\begin{split}
&J^{t,x}_1 (u^*_1(\cdot,\cdot),v^*_1(\cdot,\cdot,u^*_1(\cdot,\cdot),u^*_2(\cdot,\cdot)),u^*_2(\cdot,\cdot),v^*_2(\cdot,\cdot,u^*_1(\cdot,\cdot),u^*_2(\cdot,\cdot)))\\
&\ge J^{t,x}_1 (u_1(\cdot,\cdot),v_1(\cdot,\cdot,u_1(\cdot,\cdot),u^*_2(\cdot,\cdot)),u^*_2(\cdot,\cdot),v^*_2(\cdot,\cdot,u_1(\cdot,\cdot),u^*_2(\cdot,\cdot))),\\
\end{split}
\end{equation*}
for any $(u_1,v_1)\in \mathcal{U}_1 \times \mathcal{V}_1.$\\
\indent A similar discussion is made for player 2 based on symmetry. We can get
\begin{equation*}
\begin{split}
&J^{t,x}_2 (u^*_1(\cdot,\cdot),v^*_1(\cdot,\cdot,u^*_1(\cdot,\cdot),u^*_2(\cdot,\cdot)),u^*_2(\cdot,\cdot),v^*_2(\cdot,\cdot,u^*_1(\cdot,\cdot),u^*_2(\cdot,\cdot)))\\
&\ge J^{t,x}_2 (u_1^*(\cdot,\cdot),v_1^*(\cdot,\cdot,u_1^*(\cdot,\cdot),u_2(\cdot,\cdot)),u_2(\cdot,\cdot),v_2(\cdot,\cdot,u^*_1(\cdot,\cdot),u_2(\cdot,\cdot))),\\
\end{split}
\end{equation*}
for any $(u_2,v_2)\in \mathcal{U}_2 \times \mathcal{V}_2.$\\
\indent By Definition $\ref{Definition}$, we know $\big\{(u^*_1(\cdot,\cdot),v^*_1(\cdot,\cdot,u^*_1(\cdot,\cdot),u^*_2(\cdot,\cdot))), (u^*_2(\cdot,\cdot),v^*_2(\cdot,\cdot,u^*_1(\cdot,\cdot),
u^*_2(\cdot,\cdot)))\big\}$ defined by $(\ref{proof equilibrium definition 1})$ and $(\ref{proof equilibrium definition 2})$ is the feedback Stackelberg-Nash equilibrium. The proof is complete.

\end{proof}

\section{Application to dynamic innovation and pricing decisions in a supply chain}\label{section 3}

\hspace{0.4cm} In this section, we consider a model of the dynamic innovation and pricing decision problem in the supply chain (Song et al. \cite{SCDL2021}). In the dynamic supply chain, there is a seller and a buyer. The seller sells the product to the buyer, who eventually sells the product to the final consumers. They both need to decide on their respective innovation strategies and pricing strategies. Before going into further detail about our model, we introduce some notations as follows:\\
\begin{equation*}
\begin{split}
s, b \quad\qquad\qquad\qquad\quad &\textrm{subscripts used for seller and buyer, respectively;}\\
T\ge 0 \quad\quad\qquad\qquad\quad &\textrm{the terminal time;}\\
t \qquad\quad\qquad\qquad\quad &\textrm{time } t,\quad t\in [0,T];\\
x(t)\in [0,\infty) \qquad\qquad\quad &\textrm{goodwill of the product at time } t;\\
I_s(t)\in [0,\infty) \qquad\qquad\quad &\textrm{seller's innovation effort at time } t;\\
I_b(t)\in [0,\infty) \qquad\qquad\quad &\textrm{buyer's innovation effort at time } t;\\
w(t)\in [0,\infty) \qquad\qquad\quad &\textrm{unit wholesale price at time } t;\\
p(t)\in [0,\infty) \qquad\qquad\quad &\textrm{unit retail price at time } t;\\
C_0 \quad\qquad\qquad\qquad\quad &\textrm{unit production cost};\\
D(t)\in [0,\infty) \qquad\qquad\quad &\textrm{product demand at time } t;\\
\delta(t)\in (0,\infty) \qquad\qquad\quad &\textrm{innovation effectiveness parameter at time } t;\\
\beta_p(t)\in (0,\infty) \qquad\qquad\quad &\textrm{retail price effectiveness parameter on goodwill at time } t;\\
\beta_w(t)\in (0,\infty) \qquad\qquad\quad &\textrm{wholesale price effectiveness parameter on goodwill at time } t;\\
\beta_x\in (0,\infty) \quad\qquad\qquad\quad &\textrm{decay parameter of goodwill};\\
\alpha \in (0,\infty) \quad\qquad\qquad\quad &\textrm{base market size constant};\\
\gamma_p\in (0,\infty) \quad\qquad\qquad\quad &\textrm{retail price effectiveness parameter on product demand};\\
\gamma_w\in (0,\infty) \quad\qquad\qquad\quad &\textrm{wholesale price effectiveness parameter on product demand};\\
\gamma_x\in (0,\infty) \quad\qquad\qquad\quad &\textrm{goodwill effectiveness parameter on product demand};\\
r\in (0,\infty) \quad\qquad\qquad\quad &\textrm{discount rate}.\\
\end{split}
\end{equation*}
\indent In \cite{SCDL2021}, seller is the leader in dynamic innovation and pricing decisions. Seller first announces his wholesale price $w(t)$ and innovation effort $I_s(t)$ at time $t$. Given the seller's strategies, the buyer acts as a follower to determine his retail price $p(t)$ and innovation effort $I_b(t)$ at time $t$. However, according to \cite{BAS2010, MSS2021, TCCM2021}, we know the buyer is able to act as a leader in pricing decisions when he is a powerful member such as Walmart that has the right to dictate the seller's terms. And the seller is simply the leader in making innovation decisions. Then, based on the leadership strategies $p(t)$ and $I_s(t)$, the seller as a follower gives his wholesale price $w(t)$ and the buyer as a follower gives his innovation effort $I_b(t)$. In this case, we model the dynamic innovation and pricing decision problem as a mixed leadership differential game problem. We can use the theoretic framework developed in section $\ref{section 2}$ to study it. The state of the system is the goodwill of the product in the market at time $t$. Inspired by Tapiero \cite{T1975-1, T1975-2}, we improve the deterministic dynamic model of \cite{SCDL2021} and extend it to stochastic case. We reformulate the evolution of the state as the following SDE:
\begin{equation}\label{practical example state equation}
\left\{
             \begin{aligned}
             dx(t)&=\big[\beta_p(t) p(t)+\beta_w(t) w(t)-\beta_x x(t)+\delta(t)(I_s(t)+I_b(t))\big]dt\\
                  &\quad +\big[\beta_p(t) p(t)+\beta_w(t) w(t)+\beta_x x(t)+\delta(t)(I_s(t)+I_b(t))\big]^{\frac{1}{2}}dW(t),\  0 \le t \le T, \\
              x(0)&=x_0 \ge 0,
             \end{aligned}
\right.
\end{equation}
where $\beta_p(\cdot), \beta_w(\cdot), \delta(\cdot): [0,T] \to (0,\infty)$ are continuous functions and $\beta_x$ is a positive constant. In our context, the goodwill refers to the quality of the product as perceived by the customers. The first two terms $\beta_p(t) p(t)$ and $\beta_w(t) w(t)$ on the right side of $(\ref{practical example state equation})$ represent the effect of retail price and wholesale price on goodwill. It is common for the consumers to associate high retail price and wholesale price with high perceived quality. And the third term $\beta_x x(t)$ reflects the rate at which the goodwill decays over time. The impact of seller's and buyer's innovation efforts on the goodwill of the product is denoted by the last term $\delta(t)(I_s(t)+I_b(t))$. With the increase of innovation efforts, customers' perception of quality is enhanced. We reformulate the classical linear demand function $D(t)=\alpha-\gamma_p p(t)$. Assume that demand decreases linearly with retail price and wholesale price, and increases linearly with the goodwill. Therefore, the dynamic demand function is as follows:
\begin{equation}\label{demand function}
D(t)=\alpha-\gamma_p p(t)-\gamma_w w(t)+\gamma_x x(t).
\end{equation}
The constant $\alpha > 0$ is the potential size of the initial market. The parameters $\gamma_p, \gamma_w, \gamma_x$ of $(\ref{demand function})$ denote the direct effects of retail price, wholesale price, and goodwill on the demand, respectively. We assume that the seller's unit production cost is the constant $C_0$. According to the common assumption in the literature of marginally diminishing returns for innovation expenditures (see for e.g., \cite{DJ1988, G2011, GHX2014}), we assume that the innovation costs are quadratic in innovation effort. Therefore, the innovation expenditures of the seller and buyer at time $t$ are given by $I_s^2(t)$ and $I_b^2(t)$, respectively.\\
\indent The expected profit functional that the seller needs to maximize is
\begin{equation}\label{practical example cost functional of seller}
J_s=\mathbb E \int^T_0 e^{-rt} \big[(w(t)-C_0)D(t)-I_s^2(t)\big] dt,
\end{equation}
where $r$ is the discount rate. $(I_s(\cdot), w(\cdot))\in \mathbb{U}_s \times \mathbb{V}_s =[0,\infty)\times [0,\infty)$ are the seller's control variables. Similarly, the expected profit functional for the buyer to maximize is
\begin{equation}\label{practical example cost functional of buyer}
J_b=\mathbb E \int^T_0 e^{-rt} \big[(p(t)-w(t))D(t)-I_b^2(t)\big] dt,
\end{equation}
where $(p(\cdot), I_b(\cdot))\in \mathbb{U}_b \times \mathbb{V}_b =[0,\infty)\times [0,\infty)$ are the buyer's control variables.\\
\indent The Hamiltonian functions of the seller and buyer are, respectively,
\begin{equation}
\begin{split}
H_s(t,x,I_s,w,p,I_b,y_1,A')&= y_1 \cdot \big[\beta_p(t) p+\beta_w(t) w-\beta_x x+\delta(t)(I_s+I_b)\big]\\
&\quad +\frac{1}{2} \big[\beta_p(t) p+\beta_w(t) w+\beta_x x+\delta(t)(I_s+I_b)\big]A'\\
&\quad +e^{-rt}\big[(w-C_0)(\alpha-\gamma_p p-\gamma_w w+\gamma_x x)-I_s^2\big],
\end{split}
\end{equation}
\begin{equation}
\begin{split}
H_b(t,x,I_s,w,p,I_b,y_2,A'')&= y_2 \cdot \big[\beta_p(t) p+\beta_w(t) w-\beta_x x+\delta(t)(I_s+I_b)\big]\\
&\quad +\frac{1}{2} \big[\beta_p(t) p+\beta_w(t) w+\beta_x x+\delta(t)(I_s+I_b)\big]A''\\
&\quad +e^{-rt}\big[(p-w)(\alpha-\gamma_p p-\gamma_w w+\gamma_x x)-I_b^2\big],
\end{split}
\end{equation}
where $H_s,H_b: [0,T]\times [0,\infty)\times [0,\infty)\times [0,\infty)\times [0,\infty)\times [0,\infty)\times \mathbb{R} \times \mathbb{R} \to \mathbb{R}$. For given $(I_s, p) \in [0,\infty)\times [0,\infty)$, we first consider a static Nash game with strategies $w$ and $I_b$ at the Hamiltonian level. Define
\begin{equation}
\left\{
             \begin{aligned}
             &\Gamma^w(t,x,I_s,p,y_1,y_2,A',A'')\\
             &= \mathop{argmax}\limits_{w \in\, \mathbb{V}_s} H_s\big(t,x,I_s,w,p,\Gamma^{I_b}(t,x,I_s,p,y_1,y_2,A',A''),y_1,A'\big),\\
             &\Gamma^{I_b}(t,x,I_s,p,y_1,y_2,A',A'')\\
             &= \mathop{argmax}\limits_{I_b \in\, \mathbb{V}_b} H_b\big(t,x,I_s,\Gamma^{w}(t,x,I_s,p,y_1,y_2,A',A''),p,I_b,y_2,A''\big).
             \end{aligned}
\right.
\end{equation}
Because of
\begin{equation*}
\begin{split}
&\frac{\partial H_s}{\partial w}\big(t,x,I_s,w,p,\Gamma^{I_b}(t,x,I_s,p,y_1,y_2,A',A''),y_1,A'\big)\\
&=y_1 \beta_w(t)+\frac{1}{2}\beta_w(t)A'+e^{-rt}\big(\alpha-\gamma_p p-2\gamma_w w+\gamma_x x+C_0 \gamma_w\big), \quad t\in[0,T],
\end{split}
\end{equation*}
and
\begin{equation*}
\begin{split}
&\frac{\partial^2 H_s}{\partial w^2}\big(t,x,I_s,w,p,\Gamma^{I_b}(t,x,I_s,p,y_1,y_2,A',A''),y_1,A'\big)\\
&=-2 e^{-rt} \gamma_w <0, \quad t\in[0,T],
\end{split}
\end{equation*}
we can obtain
\begin{equation}\label{retail price response}
\begin{split}
&\Gamma^{w}(t,x,I_s,p,y_1,y_2,A',A'')\\
&=\frac{e^{rt}\big[\beta_w(t) y_1+\frac{1}{2}\beta_w(t)A'\big]+\alpha-\gamma_p p+\gamma_x x+C_0 \gamma_w}{2\gamma_w}, \quad t\in[0,T].
\end{split}
\end{equation}
By performing a similar calculation for $H_b(t,x,I_s,\Gamma^{w}\big(t,x,I_s,p,y_1,y_2,A',A''),p,I_b,y_2,A''\big)$, we can get
\begin{equation}
\Gamma^{I_b}(t,x,I_s,p,y_1,y_2,A',A'')=\frac{e^{rt}\delta(t)(2y_2+A'')}{4}, \quad t\in[0,T].
\end{equation}
\indent Considering the response in the following layer, we substitute $(\Gamma^{w}, \Gamma^{I_b})$ obtained above into $H_s, H_b$ and then play a static Nash game on the leadership strategies $I_s$ and $p$ at the Hamiltonian level. Let
\begin{equation}
\left\{
             \begin{aligned}
             &\Gamma^{I_s}(t,x,y_1,y_2,A',A'')\\
             &= \mathop{argmax}\limits_{I_s \in\, \mathbb{U}_s} H_s\big(t,x,I_s,\Gamma^{w}(t,x,I_s,\Gamma^{p}(t,x,y_1,y_2,A',A''),y_1,y_2,A',A''),\\
             &\qquad \Gamma^{p}(t,x,y_1,y_2,A',A''), \Gamma^{I_b}(t,x,I_s,\Gamma^{p}(t,x,y_1,y_2,A',A''),y_1,y_2,A',A''),y_1,A'\big),\\
             &\Gamma^{p}(t,x,y_1,y_2,A',A'')\\
             &= \mathop{argmax}\limits_{p \in\, \mathbb{U}_b} H_b\big(t,x,\Gamma^{I_s}(t,x,y_1,y_2,A',A''), \Gamma^{w}(t,x,\Gamma^{I_s}(t,x,y_1,y_2,A',A''),p,y_1,y_2,A',A''),\\
             &\qquad p,\Gamma^{I_b}(t,x,\Gamma^{I_s}(t,x,y_1,y_2,A',A''),p,y_1,y_2,A',A''),y_2,A''\big).
             \end{aligned}
\right.
\end{equation}
Similar to the calculation of $(\ref{retail price response})$, we derive
\begin{equation}\label{I-s response}
\begin{split}
\Gamma^{I_s}(t,x,y_1,y_2,A',A'')=\frac{e^{rt}\delta(t)(2y_1+A')}{4},
\end{split}
\end{equation}
\begin{equation}
\begin{split}
&\Gamma^{p}(t,x,y_1,y_2,A',A'')\\
&=\frac{1}{\gamma_p (2\gamma_w+\gamma_p)}\Big[-e^{rt}\gamma_w \beta_w(t) y_1-\frac{1}{2}e^{rt}\gamma_w \beta_w(t)A'+e^{rt}(2\gamma_w \beta_p(t)-\beta_w(t)\gamma_p)y_2\\
&\quad +\frac{1}{2}e^{rt}(2\gamma_w \beta_p(t)-\beta_w(t)\gamma_p)A''+(\gamma_w+\gamma_p)(\alpha+\gamma_x x)-C_0 {\gamma_w}^2\Big], \quad t\in[0,T].
\end{split}
\end{equation}
For the sake of simplicity of notation, we set
\begin{equation*}
\begin{split}
&K_1\triangleq-\frac{e^{rt}\gamma_w \beta_w(t)}{\gamma_p (2\gamma_w+\gamma_p)}, \qquad\quad K_2\triangleq\frac{e^{rt}(2\gamma_w \beta_p(t)-\beta_w(t)\gamma_p)}{\gamma_p (2\gamma_w+\gamma_p)},\\
&K_3\triangleq\frac{(\gamma_w +\gamma_p)\gamma_x}{\gamma_p (2\gamma_w+\gamma_p)}, \qquad\quad K_4\triangleq\frac{(\gamma_w+\gamma_p)\alpha-C_0 {\gamma_w}^2}{\gamma_p (2\gamma_w+\gamma_p)},
\end{split}
\end{equation*}
and derive
\begin{equation}\label{p response}
\begin{split}
\Gamma^{p}(t,x,y_1,y_2,A',A'')=K_1 y_1+\frac{1}{2}K_1 A'+K_2 y_2+\frac{1}{2}K_2 A''+K_3 x+K_4.
\end{split}
\end{equation}
Then, we substitute $(\Gamma^{I_s}, \Gamma^{p})$ into $(\Gamma^{w}, \Gamma^{I_b})$ and obtain
\begin{equation}
\begin{split}
&\Gamma^{w}\big(t,x,\Gamma^{I_s}(t,x,y_1,y_2,A',A''),\Gamma^{p}(t,x,y_1,y_2,A',A''),y_1,y_2,A',A''\big)\\
&=\frac{e^{rt}\beta_w(t)-\gamma_p K_1}{2\gamma_w} \cdot y_1+\frac{e^{rt}\beta_w(t)-\gamma_p K_1}{4\gamma_w} \cdot A'-\frac{\gamma_p K_2}{2\gamma_w} \cdot y_2-\frac{\gamma_p K_2}{4\gamma_w} \cdot A''\\
&\quad +\frac{\gamma_x-\gamma_p K_3}{2\gamma_w} \cdot x+\frac{\alpha+C_0\gamma_w-\gamma_p K_4}{2\gamma_w},
\end{split}
\end{equation}
\begin{equation}\label{I-b response}
\hspace{-7mm}\begin{split}
&\Gamma^{I_b}\big(t,x,\Gamma^{I_s}(t,x,y_1,y_2,A',A''),\Gamma^{p}(t,x,y_1,y_2,A',A''),y_1,y_2,A',A''\big)\\
&=\frac{e^{rt}\delta(t)}{2}\cdot y_2+\frac{e^{rt}\delta(t)}{4}\cdot A'', \quad t\in [0,T].
\end{split}
\end{equation}
Similarly, we define
\begin{equation*}
\begin{split}
&K_5\triangleq\frac{e^{rt}\beta_w(t)-\gamma_p K_1}{2\gamma_w}, \qquad\quad K_6\triangleq-\frac{\gamma_p K_2}{2\gamma_w},\\
&K_7\triangleq\frac{\gamma_x-\gamma_p K_3}{2\gamma_w}, \qquad\quad K_8\triangleq\frac{\alpha+C_0\gamma_w-\gamma_p K_4}{2\gamma_w},
\end{split}
\end{equation*}
and get
\begin{equation}\label{w response}
\begin{split}
&\Gamma^{w}\big(t,x,\Gamma^{I_s}(t,x,y_1,y_2,A',A''),\Gamma^{p}(t,x,y_1,y_2,A',A''),y_1,y_2,A',A''\big)\\
&=K_5 y_1+\frac{1}{2}K_5 A'+K_6 y_2+\frac{1}{2}K_6 A''+K_7 x+K_8.
\end{split}
\end{equation}
By taking
\[y_1=\frac{\partial V_s}{\partial x},\ y_2=\frac{\partial V_b}{\partial x},\ A'=\frac{\partial^2 V_s}{\partial x^2},\ A''=\frac{\partial^2 V_b}{\partial x^2},\]
and using $\Gamma^{I_s}, \Gamma^{w}, \Gamma^{p}, \Gamma^{I_b}$ obtained previously, we derive the HJB equations of the seller and the buyer corresponding to $(\ref{HJB-1})$ and $(\ref{HJB-2})$ in section $\ref{section 2.2}$ as follows respectively:
\begin{equation}\label{HJB of seller}
\begin{split}
&\frac{\partial V_s}{\partial t}(t,x)+\frac{\partial V_s}{\partial x} \bigg(K_9 \frac{\partial V_s}{\partial x}+K_{10} \frac{\partial V_b}{\partial x}+\frac{1}{2}K_9 \frac{\partial^2 V_s}{\partial x^2}+\frac{1}{2}K_{10} \frac{\partial^2 V_b}{\partial x^2}+K_{11} x+K_{12}\bigg)\\
&+\frac{1}{2}\times \frac{\partial^2 V_s}{\partial x^2} \bigg(K_9 \frac{\partial V_s}{\partial x}+K_{10} \frac{\partial V_b}{\partial x}+\frac{1}{2}K_9 \frac{\partial^2 V_s}{\partial x^2}+\frac{1}{2}K_{10} \frac{\partial^2 V_b}{\partial x^2}+K_{11} x+2 \beta_x x+K_{12}\bigg)\\
&+\frac{e^{-rt}}{2} \bigg(K_5 \frac{\partial V_s}{\partial x}+K_{6} \frac{\partial V_b}{\partial x}+\frac{1}{2}K_5 \frac{\partial^2 V_s}{\partial x^2}+\frac{1}{2}K_{6} \frac{\partial^2 V_b}{\partial x^2}+K_7 x+K_8-C_0\bigg)\\
&\quad \times \bigg(K_{13} \frac{\partial V_s}{\partial x}-\gamma_p K_{2} \frac{\partial V_b}{\partial x}+\frac{1}{2}K_{13} \frac{\partial^2 V_s}{\partial x^2}-\frac{1}{2}\gamma_p K_{2} \frac{\partial^2 V_b}{\partial x^2}+K_{14} x+K_{15}\bigg)\\
&-\frac{e^{rt}\delta(t)^2}{16} \Big(2\frac{\partial V_s}{\partial x}+\frac{\partial^2 V_s}{\partial x^2}\Big)^2=0,\quad t\in[0,T],
\end{split}
\end{equation}
with $V_s(T,x)=0$,
\begin{equation}\label{HJB of buyer}
\begin{split}
&\frac{\partial V_b}{\partial t}(t,x)+\frac{\partial V_b}{\partial x} \bigg(K_9 \frac{\partial V_s}{\partial x}+K_{10} \frac{\partial V_b}{\partial x}+\frac{1}{2}K_9 \frac{\partial^2 V_s}{\partial x^2}+\frac{1}{2}K_{10} \frac{\partial^2 V_b}{\partial x^2}+K_{11} x+K_{12}\bigg)\\
&+\frac{1}{2}\times \frac{\partial^2 V_b}{\partial x^2} \bigg(K_9 \frac{\partial V_s}{\partial x}+K_{10} \frac{\partial V_b}{\partial x}+\frac{1}{2}K_9 \frac{\partial^2 V_s}{\partial x^2}+\frac{1}{2}K_{10} \frac{\partial^2 V_b}{\partial x^2}+K_{11} x+2 \beta_x x+K_{12}\bigg)\\
&+\frac{e^{-rt}}{2} \bigg(K_{16} \frac{\partial V_s}{\partial x}+\big(1+\frac{\gamma_p}{2 \gamma_w}\big)K_{2} \frac{\partial V_b}{\partial x}+\frac{1}{2}K_{16} \frac{\partial^2 V_s}{\partial x^2}+\frac{1}{2}\big(1+\frac{\gamma_p}{2 \gamma_w}\big)K_{2} \frac{\partial^2 V_b}{\partial x^2}+K_{17} x\\
&\quad +K_{18}\bigg)\bigg(K_{13} \frac{\partial V_s}{\partial x}-\gamma_p K_{2} \frac{\partial V_b}{\partial x}+\frac{1}{2}K_{13} \frac{\partial^2 V_s}{\partial x^2}-\frac{1}{2}\gamma_p K_{2} \frac{\partial^2 V_b}{\partial x^2}+K_{14} x+K_{15}\bigg)\\
&-\frac{e^{rt}\delta(t)^2}{16} \Big(2 \frac{\partial V_b}{\partial x}+\frac{\partial^2 V_b}{\partial x^2}\Big)^2=0,\quad t\in[0,T],
\end{split}
\end{equation}
with $V_b(T,x)=0,$ where
\begin{equation*}
\begin{split}
&K_9\triangleq\beta_p(t)K_1-\frac{\gamma_p K_1 \beta_w(t)}{2\gamma_w}+\frac{e^{rt}{\beta_w(t)}^2}{2\gamma_w}+\frac{e^{rt} {\delta(t)}^2}{2},\\
&K_{10}\triangleq\beta_p(t)K_2-\frac{\gamma_p K_2 \beta_w(t)}{2\gamma_w}+\frac{e^{rt} {\delta(t)}^2}{2},\\
&K_{11}\triangleq\beta_p(t)K_3-\frac{\gamma_p K_3 \beta_w(t)}{2\gamma_w}+\frac{\gamma_x \beta_w(t)}{2\gamma_w}-\beta_x,\\
&K_{12}\triangleq\beta_p(t)K_4-\frac{\gamma_p K_4 \beta_w(t)}{2\gamma_w}+\frac{\beta_w(t)(\alpha+C_0 \gamma_w)}{2\gamma_w},\\
&K_{13}\triangleq-\gamma_p K_1-e^{rt} \beta_w(t),\quad K_{14}\triangleq\gamma_x-\gamma_p K_3,\\
&K_{15}\triangleq\alpha-C_0 \gamma_w-\gamma_p K_4,\quad K_{16}\triangleq\big(1+\frac{\gamma_p}{2 \gamma_w}\big)K_1-\frac{e^{rt} \beta_w(t)}{2 \gamma_w},\\
&K_{17}\triangleq\big(1+\frac{\gamma_p}{2 \gamma_w}\big)K_3-\frac{\gamma_x}{2 \gamma_w},\quad K_{18}\triangleq\big(1+\frac{\gamma_p}{2 \gamma_w}\big)K_4-\frac{\alpha+C_0 \gamma_w}{2 \gamma_w}.\\
\end{split}
\end{equation*}
\indent After some observations, we guess solutions of the following forms:
\begin{equation}\label{guess solution form}
\begin{split}
&V_s(t,x)=P_2(t)x^2+P_1(t)x+P_0(t),\\
&V_b(t,x)=N_2(t)x^2+N_1(t)x+N_0(t), \quad t\in [0,T],
\end{split}
\end{equation}
where $P_2(\cdot), P_1(\cdot), P_0(\cdot), N_2(\cdot), N_1(\cdot), N_0(\cdot)$ are continuously differentiable functions on $[0,T]$ with $P_2(T)=P_1(T)=P_0(T)=N_2(T)=N_1(T)=N_0(T)=0$.\\
\indent We substitute $(\ref{guess solution form})$ into $(\ref{HJB of seller})$ and $(\ref{HJB of buyer})$, and compare the quadratic, first, and constant terms of $x$. Then the following coupled Riccati equations are obtained:
\begin{equation}\label{riccati P-2}
\begin{split}
&\dot{P}_2(t)+2P_2(t) \Big(2K_9 P_2(t)+2K_{10} N_2(t)+K_{11}\Big)\\
&+\frac{e^{-rt}}{2}\Big(2K_5 P_2(t)+2K_{6} N_2(t)+K_{7}\Big)\times \Big(2K_{13} P_2(t)-2\gamma_p K_{2} N_2(t)+K_{14}\Big)\\
&-e^{rt}\delta(t)^2 P_2(t)^2=0,\quad t\in[0,T],
\end{split}
\end{equation}
with $P_2(T)=0$;
\begin{equation*}
\begin{split}
&\dot{P}_1(t)+2P_2(t) \Big(K_9 P_1(t)+K_{10} N_1(t)+K_9 P_2(t)+K_{10} N_2(t)+K_{12}\Big)\\
&+P_1(t) \Big(2K_9 P_2(t)+2K_{10} N_2(t)+K_{11}\Big)\\
&+P_2(t) \Big(2K_9 P_2(t)+2K_{10} N_2(t)+K_{11}+2 \beta_x\Big)\\
&+\frac{e^{-rt}}{2} \Big(2 K_5 P_2(t)+2 K_6 N_2(t)+K_7\Big)\\
&\quad \times \Big(K_{13} P_1(t)-\gamma_p K_2 N_1(t)+K_{13} P_2(t)-\gamma_p K_2 N_2(t)+K_{15}\Big)\\
\end{split}
\end{equation*}
\begin{equation}\label{riccati P-1}
\begin{split}
&+\frac{e^{-rt}}{2} \Big(K_5 P_1(t)+K_6 N_1(t)+K_5 P_2(t)+K_6 N_2(t)+K_8-C_0\Big)\\
&\quad \times \Big(2 K_{13} P_2(t)-2\gamma_p K_2 N_2(t)+K_{14}\Big)\\
&-e^{rt}\delta(t)^2 P_2(t)\big(P_1(t)+P_2(t)\big)=0,\quad t\in[0,T],
\end{split}
\end{equation}
with $P_1(T)=0$;
\begin{equation}\label{riccati P-0}
\begin{split}
&\dot{P}_0(t)+P_1(t) \Big(K_9 P_1(t)+K_{10} N_1(t)+K_9 P_2(t)+K_{10} N_2(t)+K_{12}\Big)\\
&+P_2(t) \Big(K_9 P_1(t)+K_{10} N_1(t)+K_9 P_2(t)+K_{10} N_2(t)+K_{12}\Big)\\
&+\frac{e^{-rt}}{2} \Big(K_5 P_1(t)+K_6 N_1(t)+K_5 P_2(t)+K_6 N_2(t)+K_8-C_0\Big)\\
&\quad \times \Big(K_{13} P_1(t)-\gamma_p K_2 N_1(t)+K_{13} P_2(t)-\gamma_p K_2 N_2(t)+K_{15}\Big)\\
&-\frac{e^{rt}\delta(t)^2}{4}\big(P_1(t)+P_2(t)\big)^2=0,\quad t\in[0,T],
\end{split}
\end{equation}
with $P_0(T)=0$;
\begin{equation}\label{riccati N-2}
\begin{split}
&\dot{N}_2(t)+2N_2(t) \Big(2K_9 P_2(t)+2K_{10} N_2(t)+K_{11}\Big)\\
&+\frac{e^{-rt}}{2} \Big(2K_{16} P_2(t)+2 \big(1+\frac{\gamma_p}{2\gamma_w}\big)K_2 N_2(t)+K_{17}\Big)\\
&\quad \times \Big(2K_{13} P_2(t)-2\gamma_p K_2 N_2(t)+K_{14}\Big)-e^{rt}\delta(t)^2 N_2(t)^2=0,\quad t\in[0,T],
\end{split}
\end{equation}
with $N_2(T)=0$;
\begin{equation}\label{riccati N-1}
\begin{split}
&\dot{N}_1(t)+2N_2(t) \Big(K_9 P_1(t)+K_{10} N_1(t)+K_9 P_2(t)+K_{10} N_2(t)+K_{12}\Big)\\
&+N_1(t) \Big(2K_9 P_2(t)+2K_{10} N_2(t)+K_{11}\Big)\\
&+N_2(t) \Big(2K_9 P_2(t)+2K_{10} N_2(t)+K_{11}+2 \beta_x\Big)\\
&+\frac{e^{-rt}}{2} \Big(2 K_{16} P_2(t)+2 \big(1+\frac{\gamma_p}{2\gamma_w}\big)K_2N_2(t)+K_{17}\Big)\\
&\quad \times \Big(K_{13} P_1(t)-\gamma_p K_2 N_1(t)+K_{13} P_2(t)-\gamma_p K_2 N_2(t)+K_{15}\Big)\\
&+\frac{e^{-rt}}{2} \Big(K_{16} P_1(t)+\big(1+\frac{\gamma_p}{2\gamma_w}\big)K_2 N_1(t)+K_{16} P_2(t)+\big(1+\frac{\gamma_p}{2\gamma_w}\big)K_2 N_2(t)+K_{18}\Big)\\
&\quad \times \Big(2 K_{13} P_2(t)-2\gamma_p K_2 N_2(t)+K_{14}\Big)\\
&-e^{rt}\delta(t)^2 N_2(t)\big(N_1(t)+N_2(t)\big)=0,\quad t\in[0,T],
\end{split}
\end{equation}
with $N_1(T)=0$;
\begin{equation}\label{riccati N-0}
\begin{split}
&\dot{N}_0(t)+N_1(t)\Big(K_9 P_1(t)+K_{10} N_1(t)+K_9 P_2(t)+K_{10} N_2(t)+K_{12}\Big)\\
&+N_2(t)\Big(K_9 P_1(t)+K_{10} N_1(t)+K_9 P_2(t)+K_{10} N_2(t)+K_{12}\Big)\\
&+\frac{e^{-rt}}{2}\Big(K_{16} P_1(t)+\big(1+\frac{\gamma_p}{2\gamma_w}\big)K_2 N_1(t)+K_{16} P_2(t)+\big(1+\frac{\gamma_p}{2\gamma_w}\big)K_2 N_2(t)+K_{18}\Big)\\
&\quad \times \Big(K_{13} P_1(t)-\gamma_p K_2 N_1(t)+K_{13} P_2(t)-\gamma_p K_2 N_2(t)+K_{15}\Big)\\
&-\frac{e^{rt}\delta(t)^2}{4} \big(N_1(t)+N_2(t)\big)^2=0,\quad t\in[0,T],
\end{split}
\end{equation}
with $N_0(T)=0$.\\
\indent According to the verification theorem in section $\ref{section 2.2}$ and previously obtained $\Gamma^{I_s}, \Gamma^{p}, \Gamma^{w}, \Gamma^{I_b}$, namely, $(\ref{I-s response}), (\ref{p response}), (\ref{w response}), (\ref{I-b response})$, we can use the solutions of the coupled Riccati equations $(\ref{riccati P-2})$-$(\ref{riccati N-0})$ to get the feedback Stackelberg-Nash equilibrium.

\begin{mythm}
A feedback Stackelberg-Nash equilibrium is derived in which the seller's wholesale price and innovation effort strategies are given by the following forms
\begin{equation}\label{explicit feedback equilibrium of the seller}
\begin{split}
&w^*=\Big(2K_5P_2(t)+2K_6N_2(t)+K_7\Big)x\\
&\qquad +K_5P_1(t)+K_5P_2(t)+K_6N_1(t)+K_6N_2(t)+K_8,\\
&I_s^*=\frac{e^{rt}\delta(t)}{2}\Big(2P_2(t)x+P_1(t)+P_2(t)\Big), \quad t\in [0,T],
\end{split}
\end{equation}
and the retail price and innovation effort strategies of the buyer are given by the following forms
\begin{equation}\label{explicit feedback equilibrium of the buyer}
\begin{split}
&p^*=\Big(2K_1P_2(t)+2K_2N_2(t)+K_3\Big)x\\
&\qquad +K_1P_1(t)+K_1P_2(t)+K_2N_1(t)+K_2N_2(t)+K_4,\\
&I_b^*=\frac{e^{rt}\delta(t)}{2} \Big(2N_2(t)x+N_1(t)+N_2(t)\Big), \quad t\in [0,T].
\end{split}
\end{equation}
The value functions of the two players are quadratic in the state variable and can be written in the form $(\ref{guess solution form})$, where the coefficients $P_2(\cdot), P_1(\cdot), P_0(\cdot), N_2(\cdot), N_1(\cdot), N_0(\cdot)$ are the solutions of coupled Riccati equations $(\ref{riccati P-2})$-$(\ref{riccati N-0})$ and depend on model parameters.
\end{mythm}

\begin{Remark}
Since the decisions in this supply chain problem should be nonnegative, the feedback equilibrium decisions of innovation and pricing which are finally applied to practice can be expressed as $max(w^*,0)$, $max(I_s^*,0)$, $max(p^*,0)$, and $max(I_b^*,0)$.
\end{Remark}

\indent Next, we will discuss the existence and uniqueness of the solutions to the coupled Riccati equations $(\ref{riccati P-2})$-$(\ref{riccati N-0})$. First, we consider the coupled Riccati system consisting of $(\ref{riccati P-2})$ and $(\ref{riccati N-2})$, which does not include $P_1(\cdot), P_0(\cdot), N_1(\cdot), N_0(\cdot)$. After calculation, equations $(\ref{riccati P-2})$ and $(\ref{riccati N-2})$ can be abbreviated as follows:
\begin{equation}\label{the abbreviation of riccati P-2}
\left\{
             \begin{aligned}
             &\dot{P}_2(t)+\Phi_1(t)P_2(t)^2+\Phi_2(t)N_2(t)^2+\Phi_3(t)P_2(t)N_2(t)\\
             &+\Phi_4(t)P_2(t)+\Phi_5(t)N_2(t)+\Phi_6(t)=0,\quad t\in[0,T],\\
             &P_2(T)=0;
             \end{aligned}
\right.
\end{equation}
\begin{equation}\label{the abbreviation of riccati N-2}
\left\{
             \begin{aligned}
             &\dot{N}_2(t)+\Psi_1(t)P_2(t)^2+\Psi_2(t)N_2(t)^2+\Psi_3(t)P_2(t)N_2(t)\\
             &+\Psi_4(t)P_2(t)+\Psi_5(t)N_2(t)+\Psi_6(t)=0,\quad t\in[0,T],\\
             &N_2(T)=0;
             \end{aligned}
\right.
\end{equation}
where
\begin{equation*}
\begin{split}
&\Phi_1(t)\triangleq4K_9+2e^{-rt}K_5K_{13}-e^{rt}{\delta(t)}^2,\quad \Phi_2(t)\triangleq-2e^{-rt}\gamma_p K_2 K_6,\\
&\Phi_3(t)\triangleq4K_{10}-2e^{-rt}\gamma_p K_2 K_5+2e^{-rt}K_6K_{13},\\
&\Phi_4(t)\triangleq2K_{11}+e^{-rt}K_5K_{14}+e^{-rt}K_7K_{13},\\
&\Phi_5(t)\triangleq e^{-rt}K_6K_{14}-e^{-rt}\gamma_p K_2K_7,\quad \Phi_6(t)\triangleq\frac{e^{-rt}}{2}K_7K_{14},\\
&\Psi_1(t)\triangleq2e^{-rt}K_{13}K_{16},\quad \Psi_2(t)\triangleq4K_{10}-2e^{-rt}\big(1+\frac{\gamma_p}{2\gamma_w}\big)\gamma_p {K_2}^2-e^{rt}{\delta(t)}^2,\\
&\Psi_3(t)\triangleq4K_9-2e^{-rt}\gamma_pK_2K_{16}+2e^{-rt}\big(1+\frac{\gamma_p}{2\gamma_w}\big) K_2K_{13},\\
&\Psi_4(t)\triangleq e^{-rt}K_{14}K_{16}+e^{-rt}K_{13}K_{17},\\
&\Psi_5(t)\triangleq2K_{11}+e^{-rt}\big(1+\frac{\gamma_p}{2\gamma_w}\big) K_2K_{14}-e^{-rt}\gamma_pK_2K_{17},\quad \Psi_6(t)\triangleq\frac{e^{-rt}}{2}K_{14}K_{17}.\\
\end{split}
\end{equation*}
By setting
\begin{equation}
\begin{split}
\widetilde{P}_2(t)=P_2(T-t),\quad \widetilde{N}_2(t)=N_2(T-t),\quad t\in[0,T],
\end{split}
\end{equation}
we get $(\ref{the abbreviation of riccati P-2})$ and $(\ref{the abbreviation of riccati N-2})$ are equivalent to
\begin{equation}\label{riccati tilde P-2}
\left\{
             \begin{aligned}
             &\dot{\widetilde{P}}_2(t)=\widetilde{\Phi}_1(t)\widetilde{P}_2(t)^2+\widetilde{\Phi}_2(t)\widetilde{N}_2(t)^2+\widetilde{\Phi}_3(t)\widetilde{P}_2(t)\widetilde{N}_2(t)
             +\widetilde{\Phi}_4(t)\widetilde{P}_2(t)\\
             & +\widetilde{\Phi}_5(t)\widetilde{N}_2(t)+\widetilde{\Phi}_6(t),\quad t\in[0,T],\\
             &\widetilde{P}_2(0)=0;
             \end{aligned}
\right.
\end{equation}
\begin{equation}\label{riccati tilde N-2}
\left\{
             \begin{aligned}
             &\dot{\widetilde{N}}_2(t)=\widetilde{\Psi}_1(t)\widetilde{P}_2(t)^2+\widetilde{\Psi}_2(t)\widetilde{N}_2(t)^2+\widetilde{\Psi}_3(t)\widetilde{P}_2(t)\widetilde{N}_2(t)
             +\widetilde{\Psi}_4(t)\widetilde{P}_2(t)\\
             & +\widetilde{\Psi}_5(t)\widetilde{N}_2(t)+\widetilde{\Psi}_6(t),\quad t\in[0,T],\\
             &\widetilde{N}_2(0)=0;
             \end{aligned}
\right.
\end{equation}
where
\[\widetilde{\Phi}_i(t)=\Phi_i(T-t),\qquad \widetilde{\Psi}_i(t)=\Psi_i(T-t), \quad t\in[0,T].\]
Let
\begin{equation}
\begin{split}
M(t)\triangleq\left( \begin{array}{ccc}
\widetilde{P}_2(t)\\
\widetilde{N}_2(t)\\
\end{array} \right), \quad t\in [0,T],
\end{split}
\end{equation}
and $(\ref{riccati tilde P-2}), (\ref{riccati tilde N-2})$ can be written as
\begin{equation}\label{equation M}
\begin{split}
\dot{M}(t)=G(t,M(t)), \quad t\in [0,T],
\end{split}
\end{equation}
where
\begin{equation*}
\begin{split}
&G(t,M(t))\triangleq\left( \begin{array}{ccc}
G_1(t,M(t))\\
G_2(t,M(t))\\
\end{array} \right),\\
&G_1(t,M(t))\triangleq\widetilde{\Phi}_1(t)M(t)^\top M(t)+\left( \begin{array}{ccc}
\widetilde{\Phi}_4(t) & \widetilde{\Phi}_5(t)
\end{array} \right)M(t)\\
&\qquad\qquad\qquad+\left( \begin{array}{ccc}
0 & 1
\end{array} \right)M(t)M(t)^\top \left( \begin{array}{ccc}
\widetilde{\Phi}_3(t)\\
\widetilde{\Phi}_2(t)-\widetilde{\Phi}_1(t)\\
\end{array} \right)+\widetilde{\Phi}_6(t),\\
&G_2(t,M(t))\triangleq\widetilde{\Psi}_1(t)M(t)^\top M(t)+\left( \begin{array}{ccc}
\widetilde{\Psi}_4(t) & \widetilde{\Psi}_5(t)
\end{array} \right)M(t)\\
&\qquad\qquad\qquad+\left( \begin{array}{ccc}
0 & 1
\end{array} \right)M(t)M(t)^\top \left( \begin{array}{ccc}
\widetilde{\Psi}_3(t)\\
\widetilde{\Psi}_2(t)-\widetilde{\Psi}_1(t)\\
\end{array} \right)+\widetilde{\Psi}_6(t),\\
\end{split}
\end{equation*}
with
\begin{equation}\label{the initial value of M}
\begin{split}
M(0)=\left( \begin{array}{ccc}
0\\
0\\
\end{array} \right).
\end{split}
\end{equation}
\indent Since $\beta_p(t), \beta_w(t), \delta(t)$ are continuous functions on $[0,T]$, then $G(t,M):[0,T]\times \mathbb{R}^2 \to \mathbb{R}^2$ is continuous in $t$ and uniformly locally Lipschitz continuous in $M$. According to Picard-Lindel{\"o}f's Theorem (See, for example, Theorem 15.2 in Agarwal and O'Regan \cite{AO2008}), we know that there exists an interval $[0,\eta]\subseteq [0,T]$ such that the initial value problem $(\ref{equation M})$ and $(\ref{the initial value of M})$ has a unique solution on $[0,\eta]$. \\
\indent When the solutions of the coupled Riccati equations $(\ref{riccati P-2})$ and $(\ref{riccati N-2})$ are known, equations $(\ref{riccati P-1})$ and $(\ref{riccati N-1})$ become a system of linear ODEs with respect to $P_1(\cdot)$ and $N_1(\cdot)$. Since $\beta_p(t), \beta_w(t), \delta(t)$ are continuous on $[0,T]$, there exists a unique solution $(P_1(\cdot), N_1(\cdot))$ to equations $(\ref{riccati P-1})$ and $(\ref{riccati N-1})$. Finally, substituting the solutions of equations $(\ref{riccati P-2}), (\ref{riccati P-1}), (\ref{riccati N-2}),$ and $(\ref{riccati N-1})$ into $(\ref{riccati P-0})$ and $(\ref{riccati N-0})$, $P_0(\cdot)$ and $N_0(\cdot)$ can be solved easily.

\section{Numerical analysis}\label{section 4}

\hspace{0.4cm} Since explicit solutions to the system of coupled Riccati equations $(\ref{riccati P-2})$-$(\ref{riccati N-0})$ are difficult to obtain in general, a numerical analysis is performed to understand the dependence of decision variables on representative model parameters $C_0$ and $\delta(t)$. We fix the values of all parameters except one parameter of interest, and then vary the value of that parameter to see how the feedback Stackelberg-Nash equilibrium varies with that parameter given any goodwill level $x$. According to $(\ref{explicit feedback equilibrium of the seller})$ and $(\ref{explicit feedback equilibrium of the buyer})$, we know that the feedback Stackelberg-Nash equilibrium $w^*, I_s^*, p^*$ and $I_b^*$ are linear in state $x$. We denote $w_{x}^*=\frac{\partial w^*}{\partial x}, I_{sx}^*=\frac{\partial I_s^*}{\partial x}, p_{x}^*=\frac{\partial p^*}{\partial x}, I_{bx}^*=\frac{\partial I_b^*}{\partial x}$. At the same time, $w_{0}^*, I_{s0}^*, p_{0}^*$ and $I_{b0}^*$ denote constant terms in the optimal policies $w^*, I_s^*, p^*$ and $I_b^*$, respectively. Therefore, the feedback Stackelberg-Nash equilibrium can be written as $w^*=w_{x}^* \cdot x+w_{0}^*, I_s^*=I_{sx}^* \cdot x+I_{s0}^*, p^*=p_{x}^* \cdot x+p_{0}^*, I_b^*=I_{bx}^* \cdot x+I_{b0}^*$, where the values of $w_{x}^*, I_{sx}^*, p_{x}^*, I_{bx}^*, w_{0}^*, I_{s0}^*, p_{0}^*, I_{b0}^*$ only depend on the model parameters and time variable $t$.

Figures $\ref{fig:C0-w}, \ref{fig:C0-p}, \ref{fig:C0-Is}, \ref{fig:C0-Ib}$ shows the changes of these coefficients when parameters $\beta_p(t)=0.1, \beta_w(t)=0.2, \beta_x=0.1, \delta(t)=0.1, \gamma_p=0.1, \gamma_w=0.0001, \gamma_x=0.1, \alpha=1, T=1, r=0.05$ are fixed for any $t\in[0,T]$, and unit production cost $C_0$ takes different values in $[1, 1.5, 2, 2.5]$. According to Figures $\ref{fig:C0-w}, \ref{fig:C0-p}, \ref{fig:C0-Is}, \ref{fig:C0-Ib}$, it can be seen that $w_{x}^*, p_{x}^*, I_{sx}^*, I_{bx}^*$ don't change as unit production cost $C_0$ increases. Therefore, with unit production cost $C_0$ increasing, the changes of $w_{0}^*, p_{0}^*, I_{s0}^*, I_{b0}^*$ respectively represent the corresponding changes of $w^*, p^*, I_s^*, I_b^*$.

\begin{figure}[H]
\centering

\subfigure[The impact of $C_0$ on $w^*_{x}$]
{
    \begin{minipage}{7cm}
    \centering
    \includegraphics[scale=0.4]{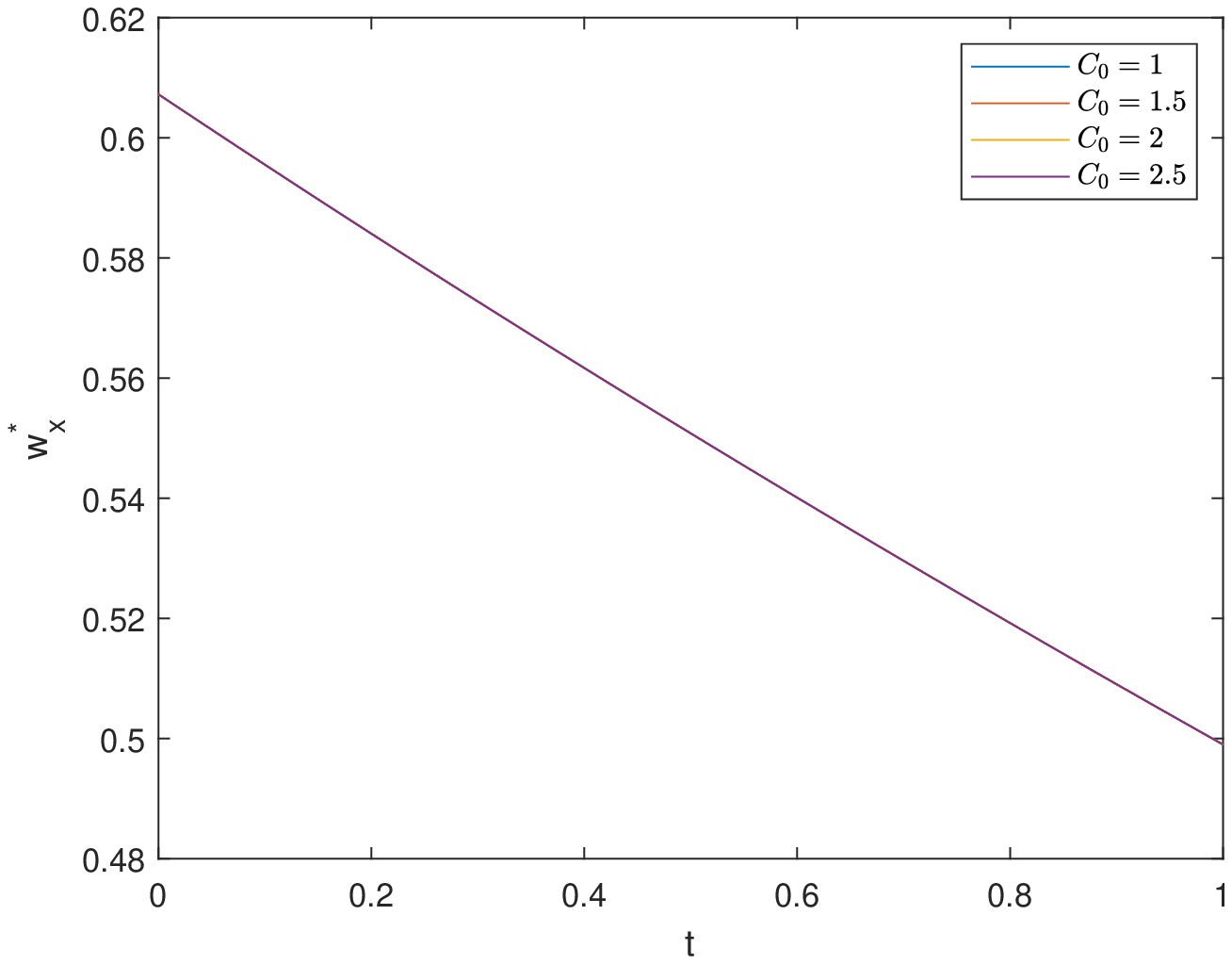}
    \end{minipage}
}
\subfigure[The impact of $C_0$ on $w^*_{0}$]
{
    \begin{minipage}{7cm}
    \centering
    \includegraphics[scale=0.4]{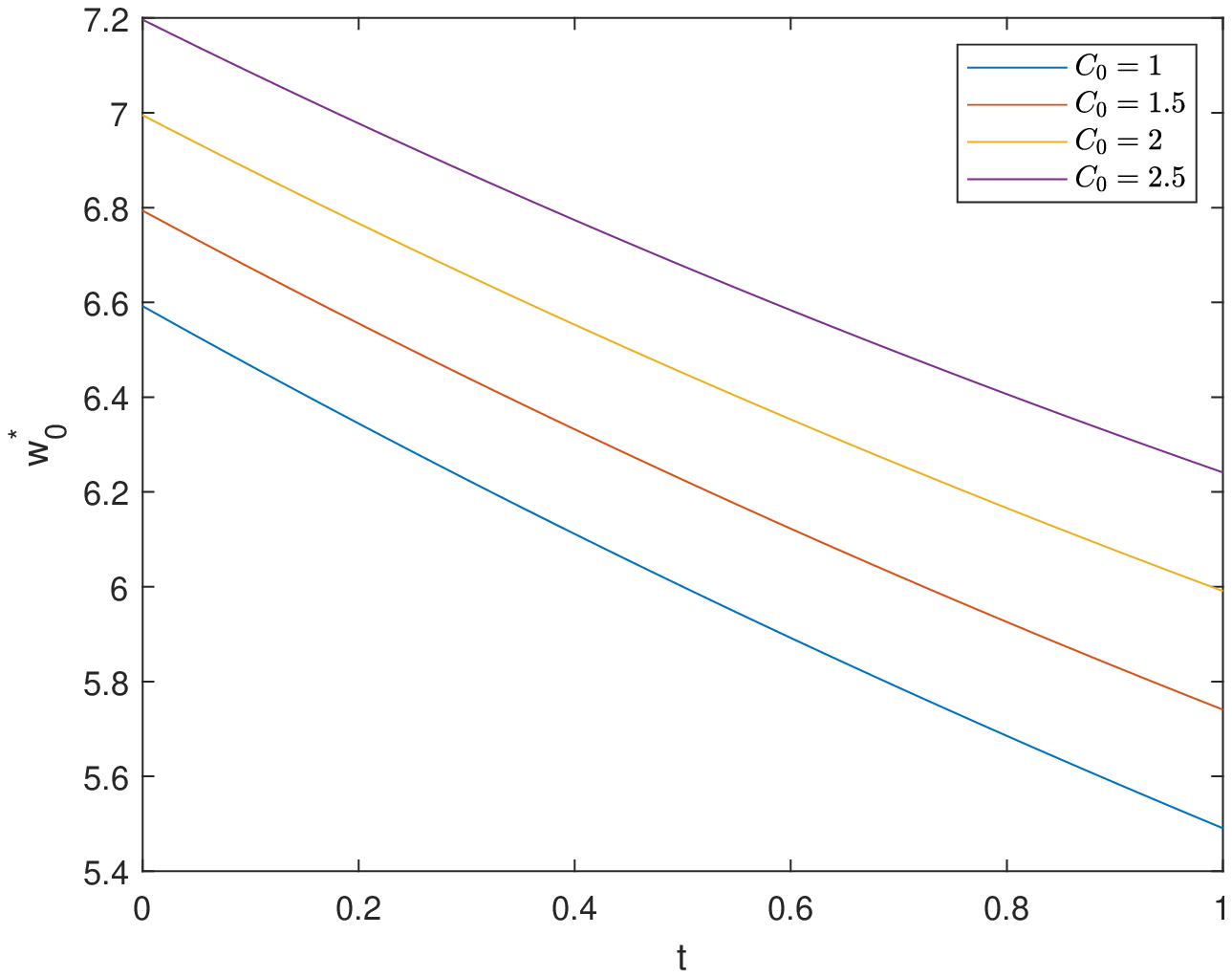}
    \end{minipage}
}
\caption{The impact of $C_0$ on $w^*$}
\label{fig:C0-w}
\end{figure}

\begin{figure}[H]
\centering

\subfigure[The impact of $C_0$ on $p^*_{x}$]
{
    \begin{minipage}{7cm}
    \centering
    \includegraphics[scale=0.4]{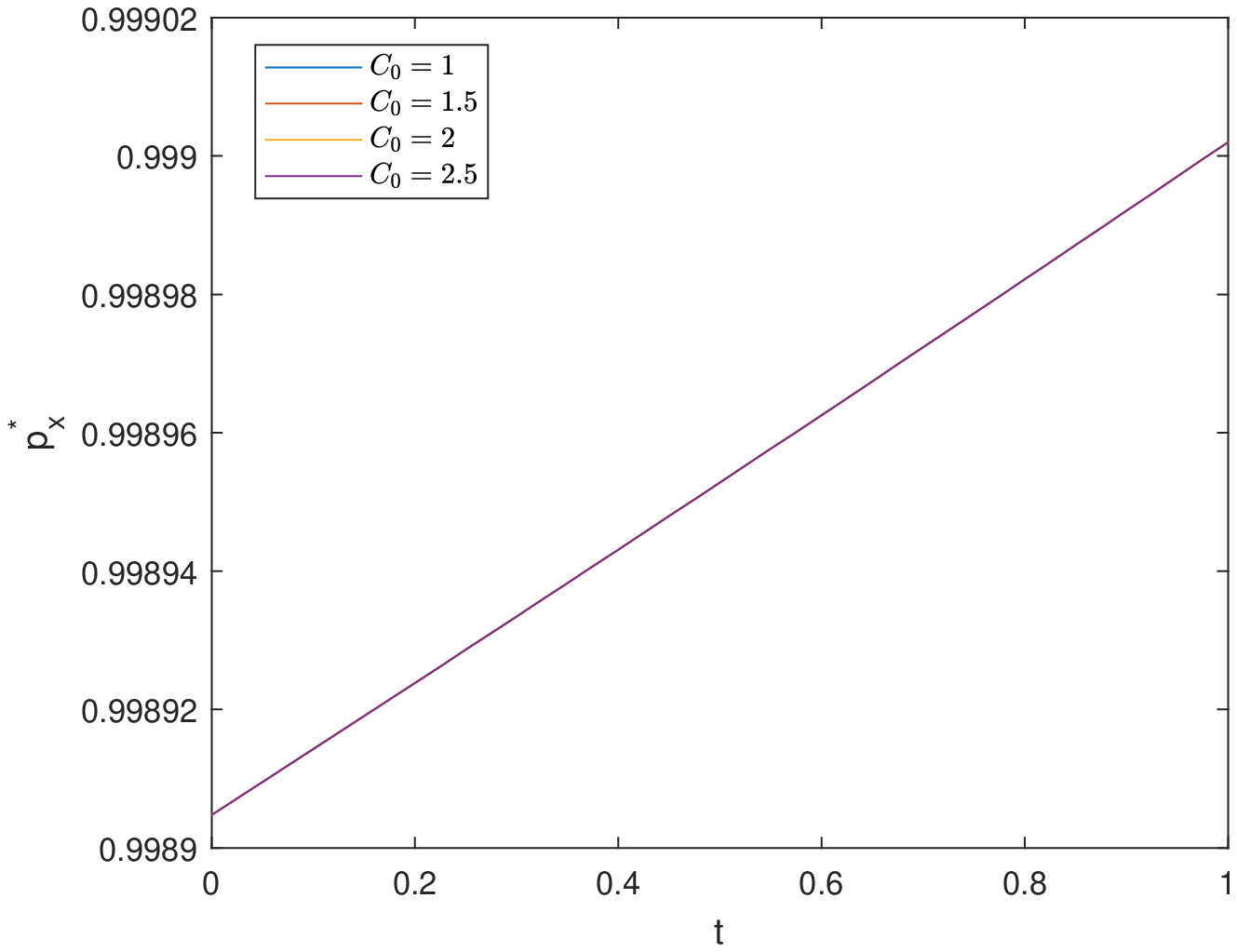}
    \end{minipage}
}
\subfigure[The impact of $C_0$ on $p^*_{0}$]
{
    \begin{minipage}{7cm}
    \centering
    \includegraphics[scale=0.4]{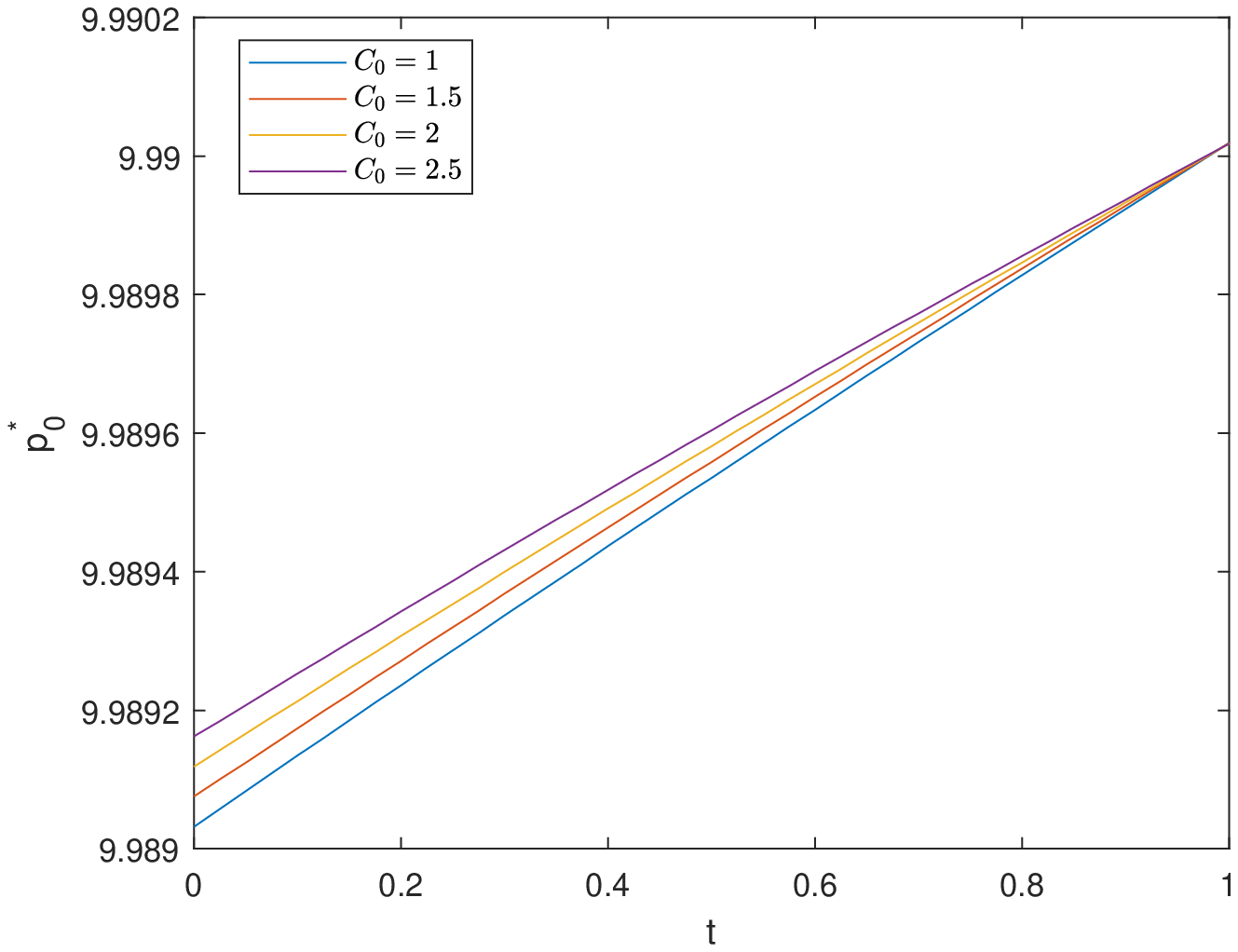}
    \end{minipage}
}
\caption{The impact of $C_0$ on $p^*$}
\label{fig:C0-p}
\end{figure}

\begin{figure}[H]
\centering

\subfigure[The impact of $C_0$ on $I^*_{sx}$]
{
    \begin{minipage}{7cm}
    \centering
    \includegraphics[scale=0.4]{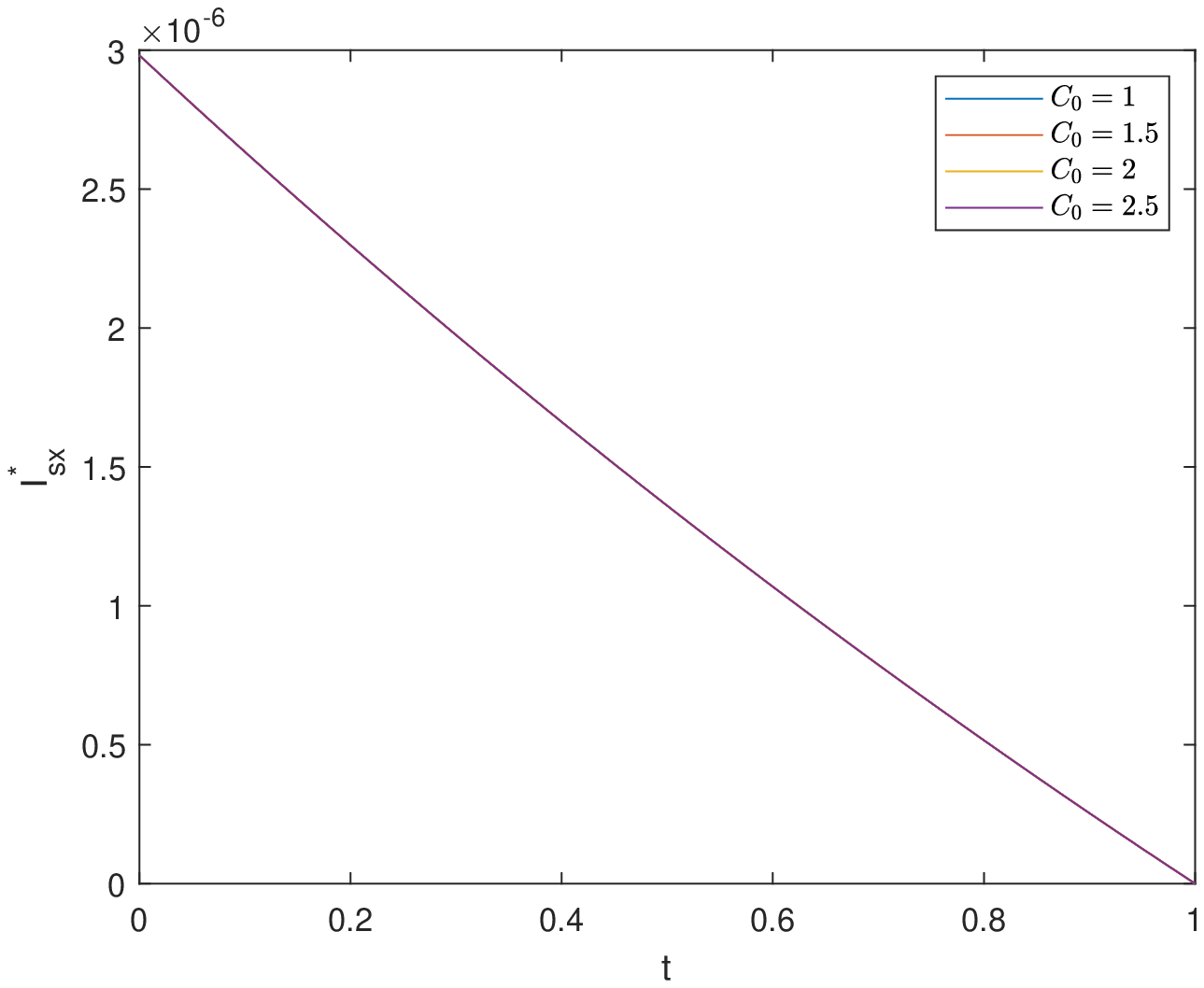}
    \end{minipage}
}
\subfigure[The impact of $C_0$ on $I^*_{s0}$]
{
    \begin{minipage}{7cm}
    \centering
    \includegraphics[scale=0.4]{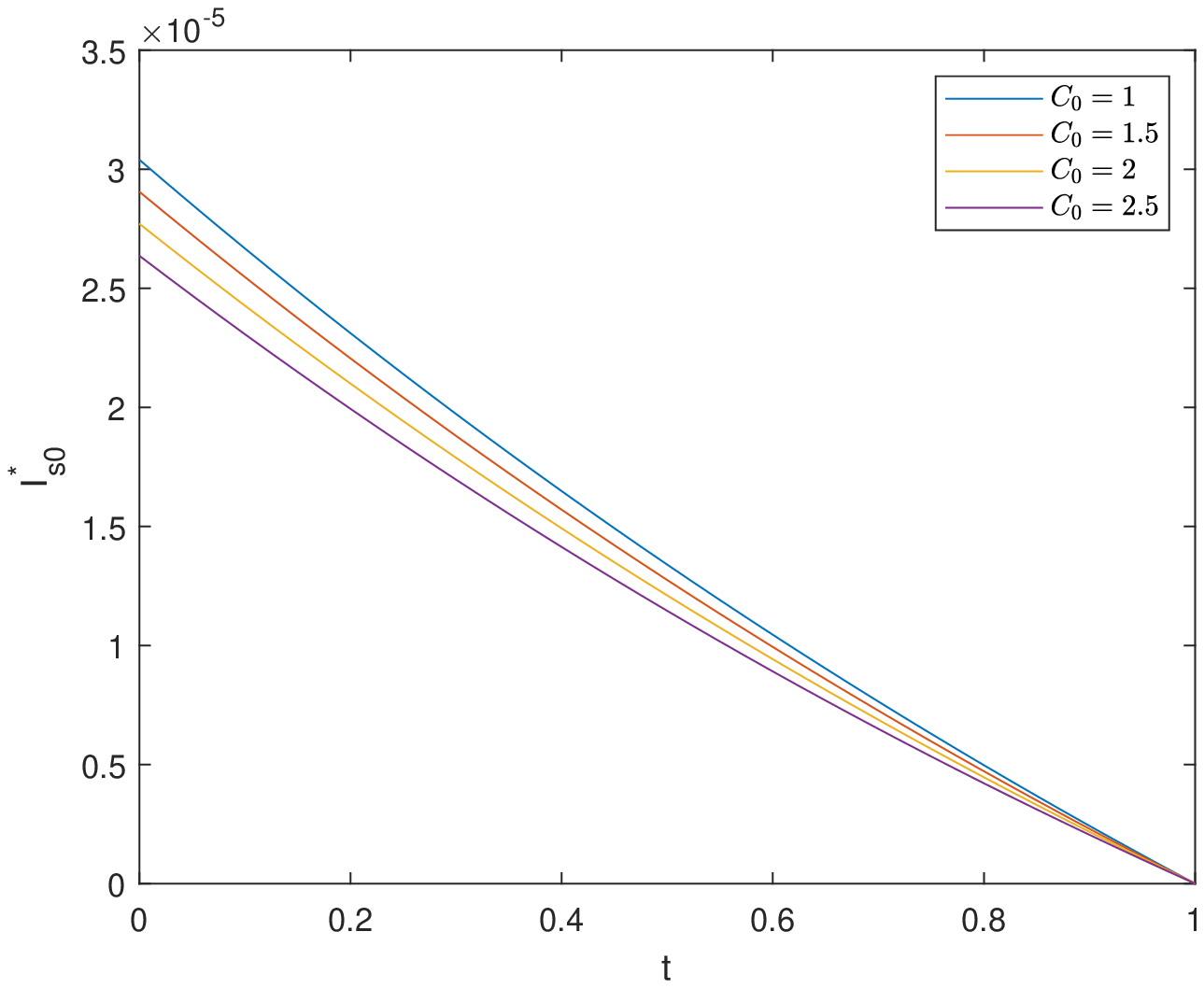}
    \end{minipage}
}
\caption{The impact of $C_0$ on $I^*_{s}$}
\label{fig:C0-Is}
\end{figure}

\begin{figure}[H]
\centering

\subfigure[The impact of $C_0$ on $I^*_{bx}$]
{
    \begin{minipage}{7cm}
    \centering
    \includegraphics[scale=0.4]{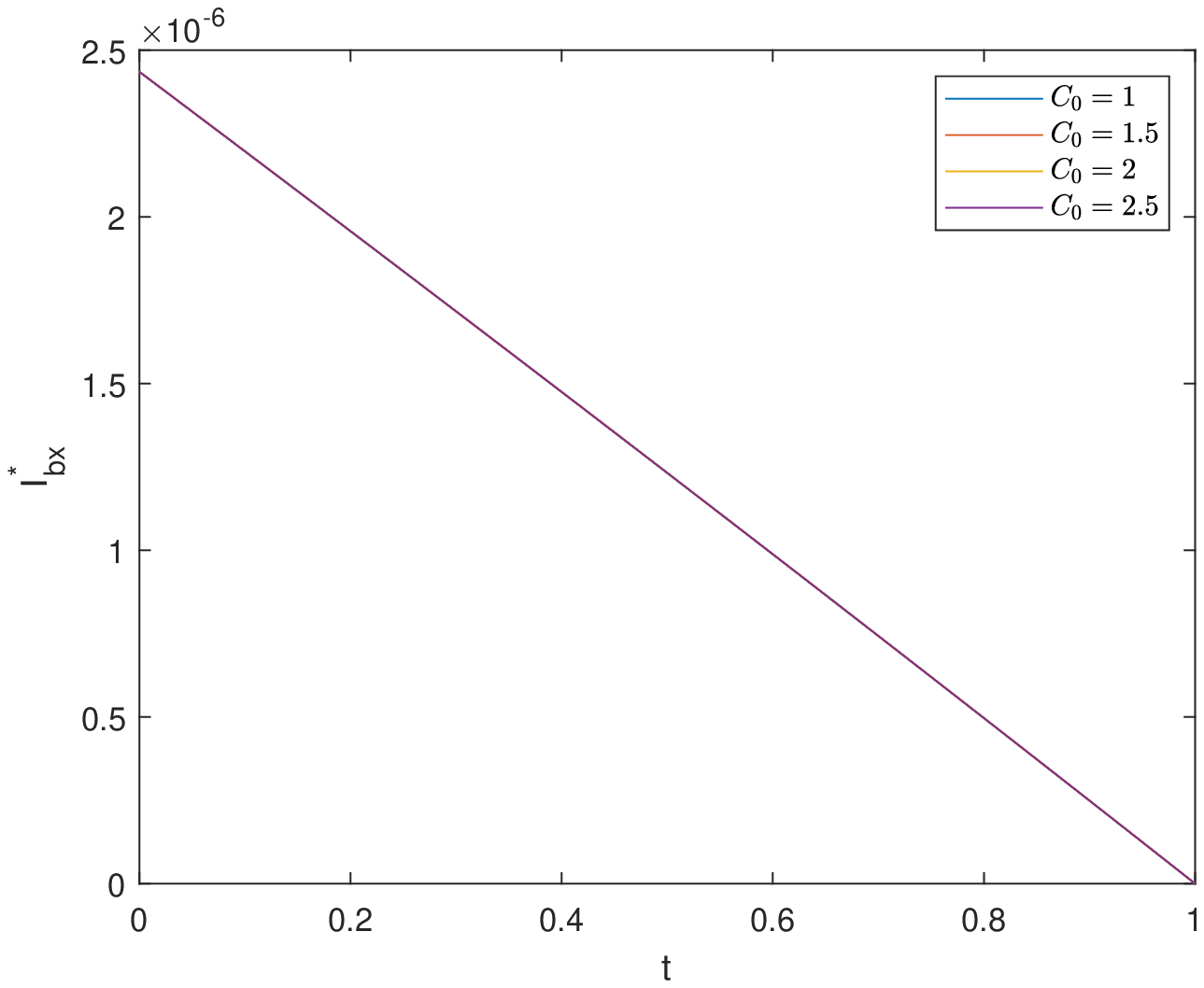}
    \end{minipage}
}
\subfigure[The impact of $C_0$ on $I^*_{b0}$]
{
    \begin{minipage}{7cm}
    \centering
    \includegraphics[scale=0.4]{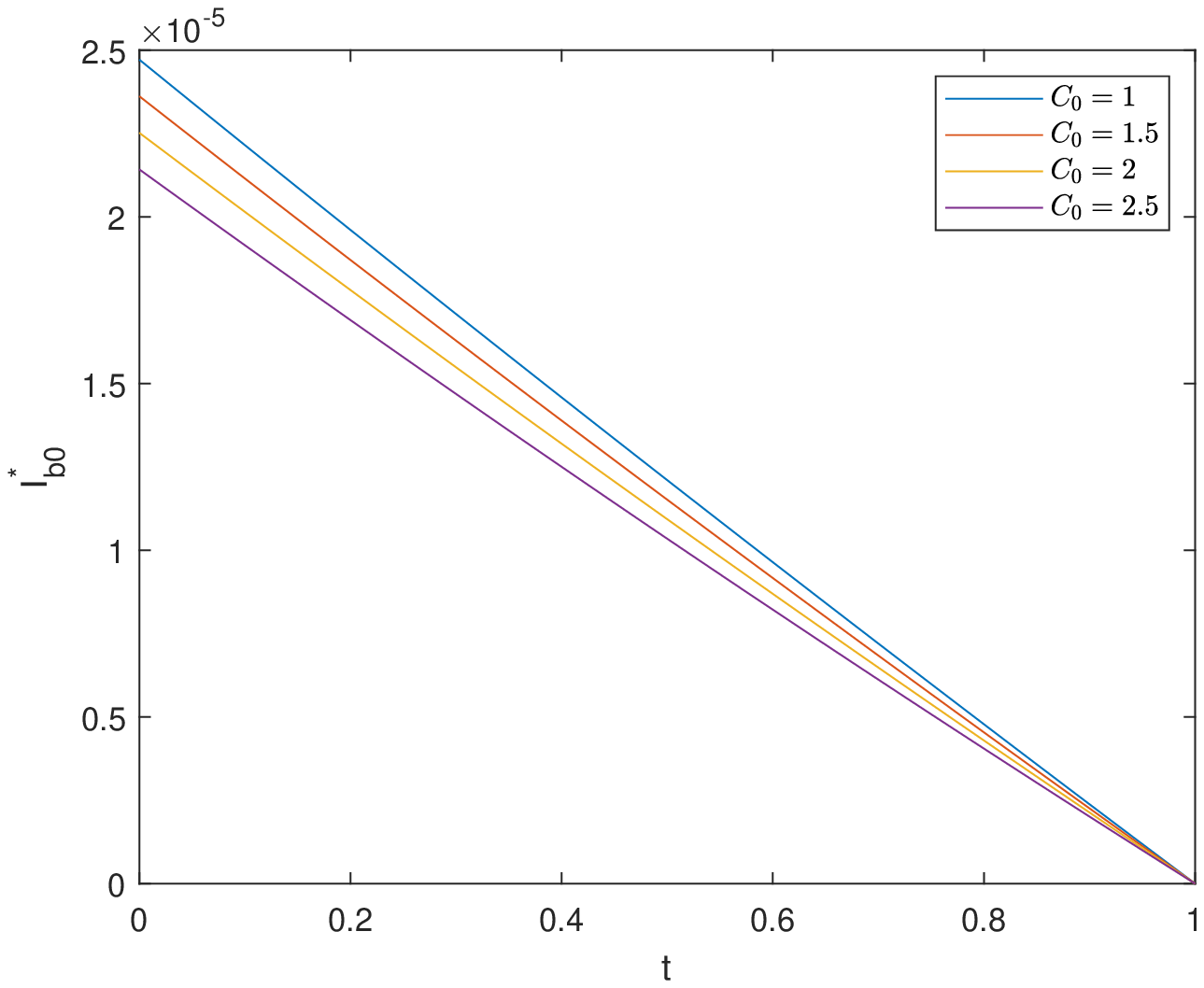}
    \end{minipage}
}
\caption{The impact of $C_0$ on $I^*_{b}$}
\label{fig:C0-Ib}
\end{figure}

When unit production cost $C_0$ increases, it is natural for the seller to raise the wholesale price (Figure $\ref{fig:C0-w}$). In our supply chain model, we consider the case where the buyer is a large retailer like Walmart that has the power to determine the terms of the seller. Such a buyer acts as the leader in the pricing strategies, and sets the retail price based on the researched market situation and consumer acceptance. There are usually many different products of the same type in the market, so the overall market situation will not change greatly due to the increase of the cost of a single product. Therefore, the change of the retail price in Figure $\ref{fig:C0-p}$ is meaningful. As the product cost $C_0$ increases, the buyer's retail price increases only slightly at the beginning. Then, as time $t$ goes to the terminal time 1, the market tends to be stable and the retail prices corresponding to different product costs tend to be equal due to the ongoing game between the seller and the buyer (Figure $\ref{fig:C0-p}$). Because the seller and the buyer need to invest capital for innovation, when the production cost $C_0$ increases, they generally choose to reduce the innovation efforts appropriately in order not to bear too much burden from excessive investment capital. Thus, the trend that the seller's and the buyer's innovation effort decrease with the increase of the product cost $C_0$ is consistent with the actual situation in the supply chain (Figure $\ref{fig:C0-Is}$ and Figure $\ref{fig:C0-Ib}$). And since the seller and the buyer don't need to put in any more innovation effort at the terminal moment, $I_{sx}^*(1)=0, I_{s0}^*(1)=0, I_{bx}^*(1)=0$, and $I_{b0}^*(1)=0$ are reasonable.

When for any $t\in[0,T]$, parameters $\beta_p(t)=0.1, \beta_w(t)=0.2, \beta_x=0.1, \gamma_p=0.1, \gamma_w=0.0001, \gamma_x=0.1, C_0=1, \alpha=1, T=1, r=0.05$ are fixed and innovation effectiveness parameter $\delta(t)$ takes different values in $[0.1, 0.2, 0.3, 0.4]$, Figures $\ref{fig:delta-w}, \ref{fig:delta-p}, \ref{fig:delta-Is}, \ref{fig:delta-Ib}$ shows the variations of coefficients $w_{x}^*, w_{0}^*, p_{x}^*, p_{0}^*, I_{sx}^*, I_{s0}^*, I_{bx}^*, I_{b0}^*$. From Figure $\ref{fig:delta-w}$ and Figure $\ref{fig:delta-p}$, we can know that the seller's wholesale price and the buyer's retail price do not change with the increase of $\delta(t)$. As the innovation effectiveness parameter $\delta(t)$ increases, the seller and the buyer have higher incentives to invest more innovation efforts. Thus, the seller's and the buyer's innovation effort increase with the value of $\delta(t)$ increasing (Figure $\ref{fig:delta-Is}$ and Figure $\ref{fig:delta-Ib}$).

\begin{figure}[H]
\centering

\subfigure[The impact of $\delta(t)$ on $w^*_{x}$]
{
    \begin{minipage}{7cm}
    \centering
    \includegraphics[scale=0.4]{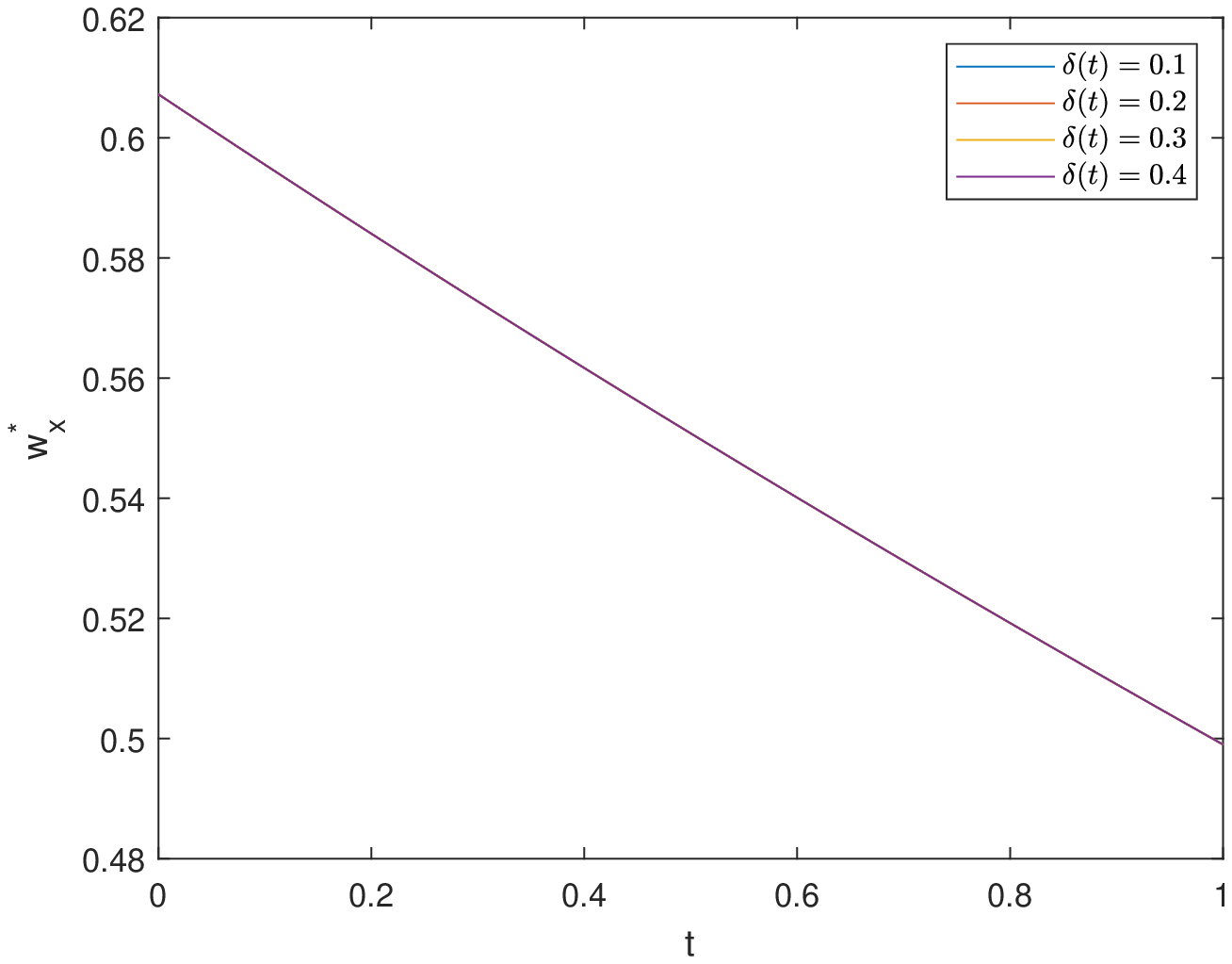}
    \end{minipage}
}
\subfigure[The impact of $\delta(t)$ on $w^*_{0}$]
{
    \begin{minipage}{7cm}
    \centering
    \includegraphics[scale=0.4]{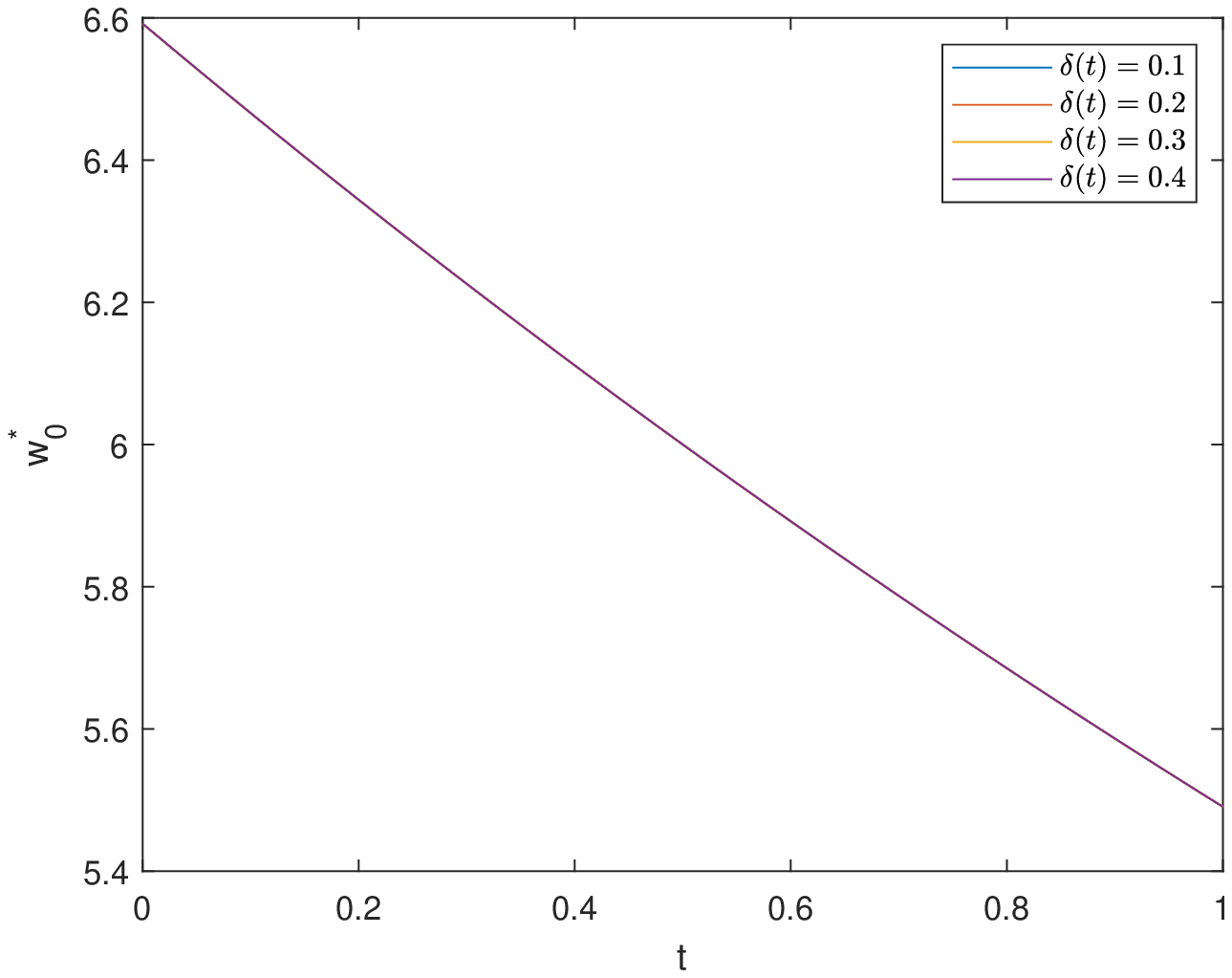}
    \end{minipage}
}
\caption{The impact of $\delta(t)$ on $w^*$}
\label{fig:delta-w}
\end{figure}

\begin{figure}[H]
\centering

\subfigure[The impact of $\delta(t)$ on $p^*_{x}$]
{
    \begin{minipage}{7cm}
    \centering
    \includegraphics[scale=0.4]{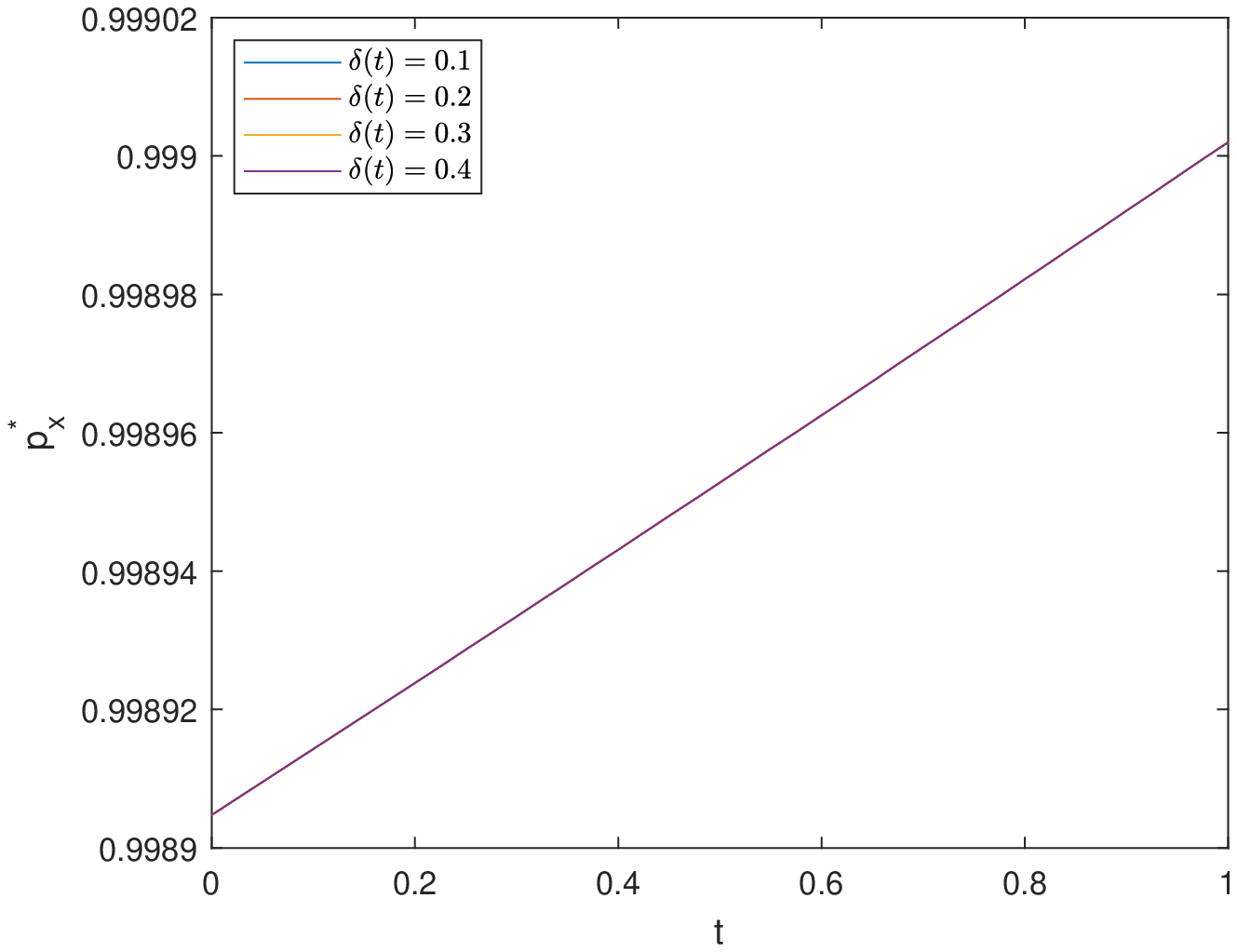}
    \end{minipage}
}
\subfigure[The impact of $\delta(t)$ on $p^*_{0}$]
{
    \begin{minipage}{7cm}
    \centering
    \includegraphics[scale=0.4]{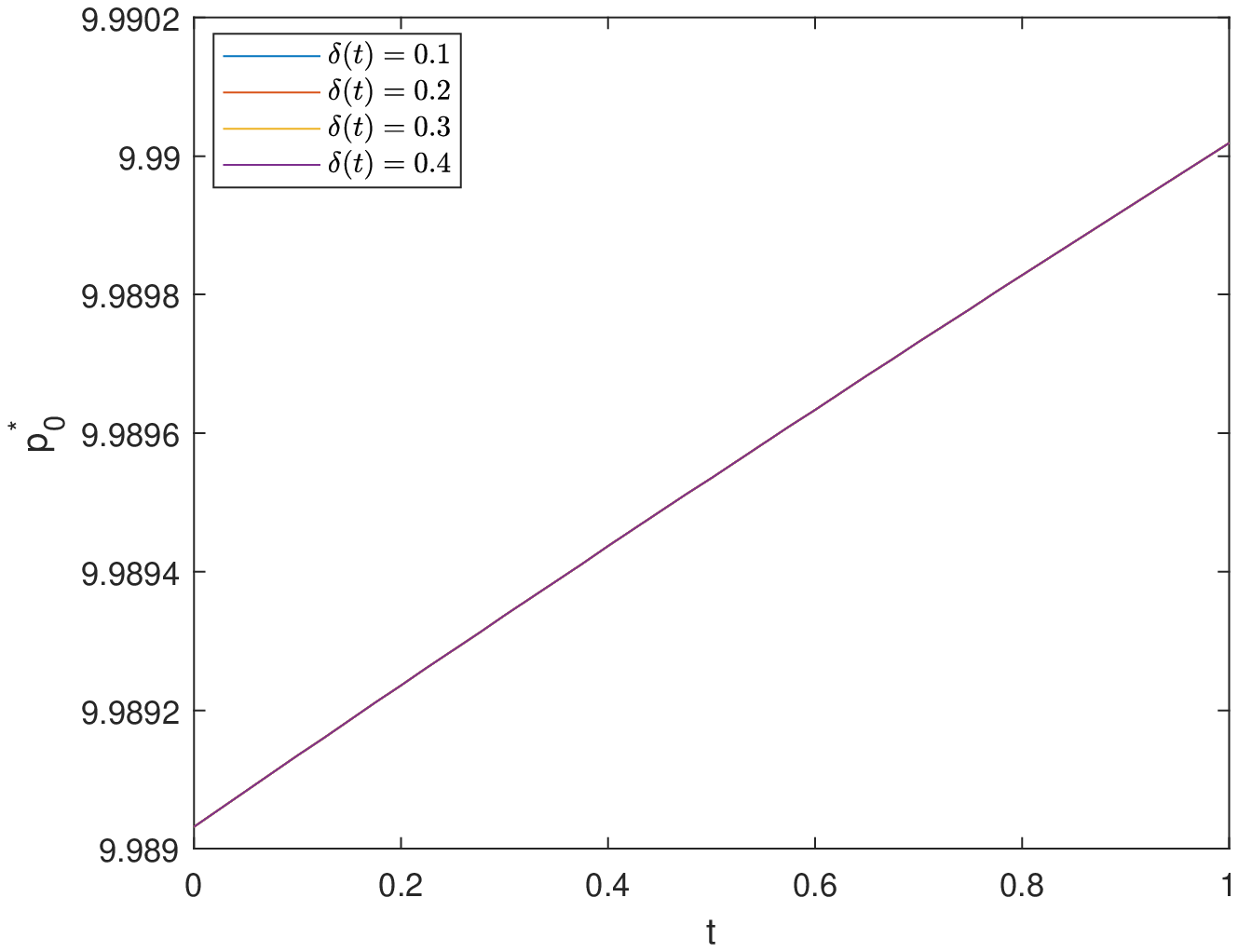}
    \end{minipage}
}
\caption{The impact of $\delta(t)$ on $p^*$}
\label{fig:delta-p}
\end{figure}

\begin{figure}[H]
\centering

\subfigure[The impact of $\delta(t)$ on $I^*_{sx}$]
{
    \begin{minipage}{7cm}
    \centering
    \includegraphics[scale=0.4]{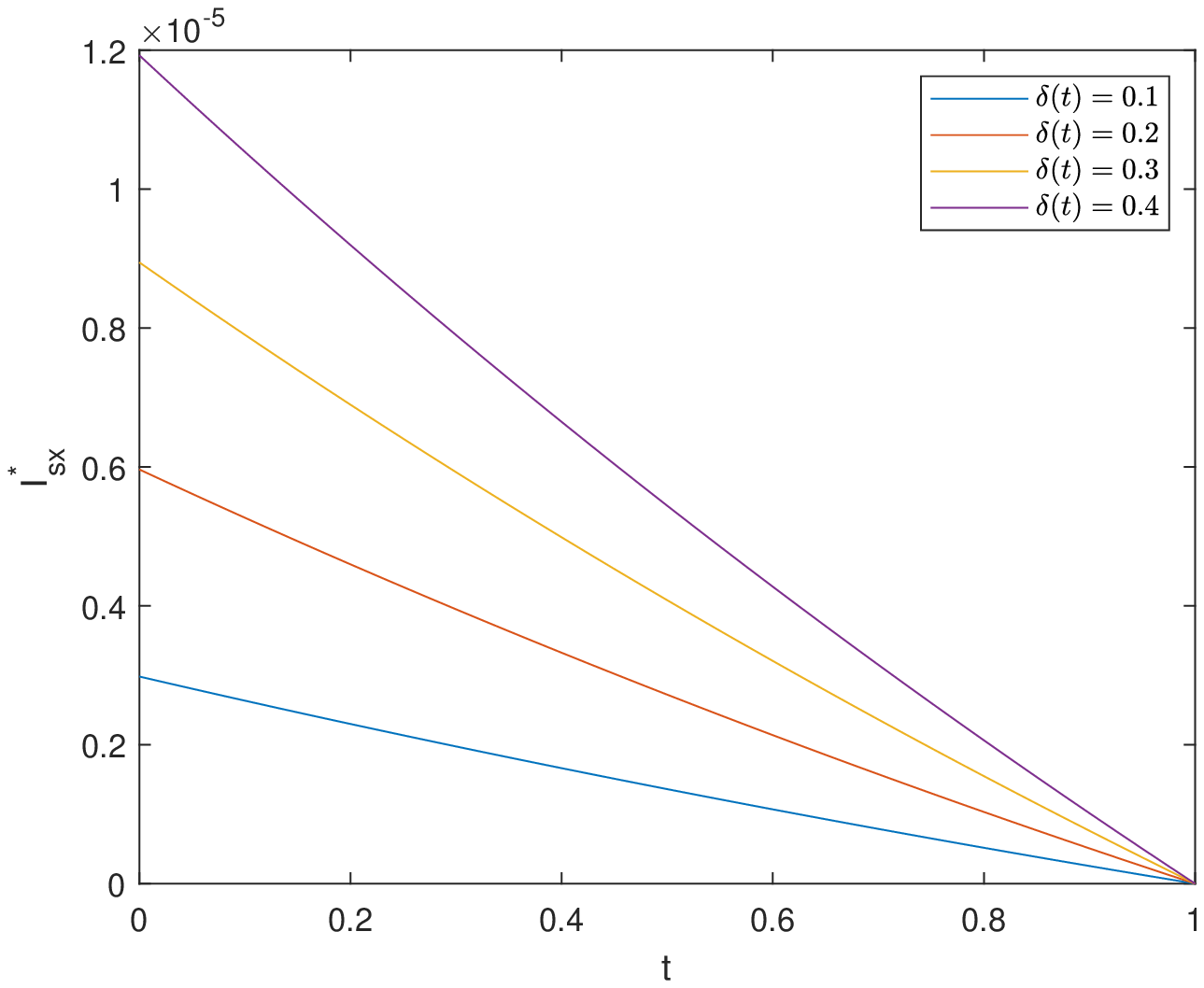}
    \end{minipage}
}
\subfigure[The impact of $\delta(t)$ on $I^*_{s0}$]
{
    \begin{minipage}{7cm}
    \centering
    \includegraphics[scale=0.4]{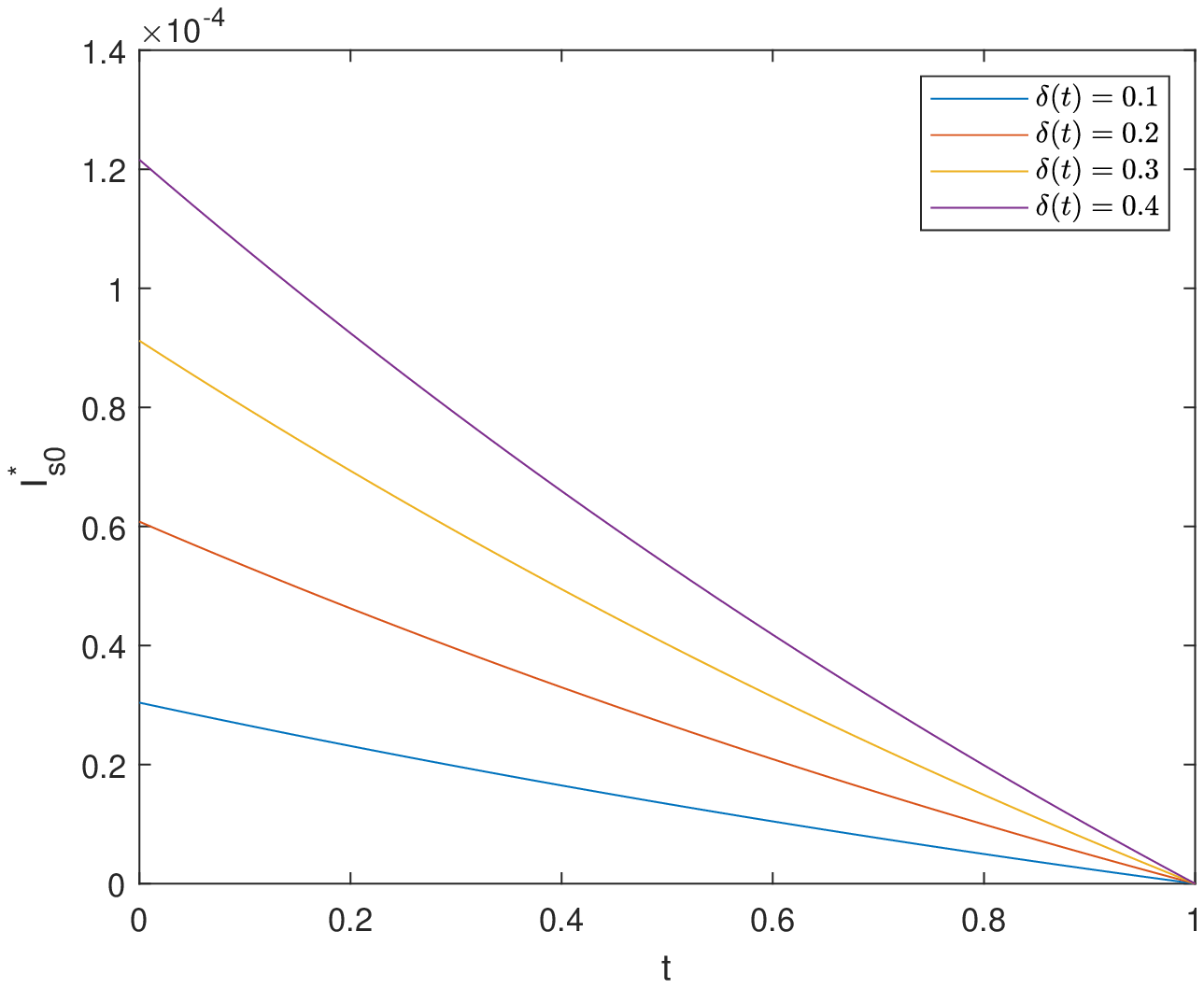}
    \end{minipage}
}
\caption{The impact of $\delta(t)$ on $I^*_s$}
\label{fig:delta-Is}
\end{figure}

\begin{figure}[H]
\centering

\subfigure[The impact of $\delta(t)$ on $I^*_{bx}$]
{
    \begin{minipage}{7cm}
    \centering
    \includegraphics[scale=0.4]{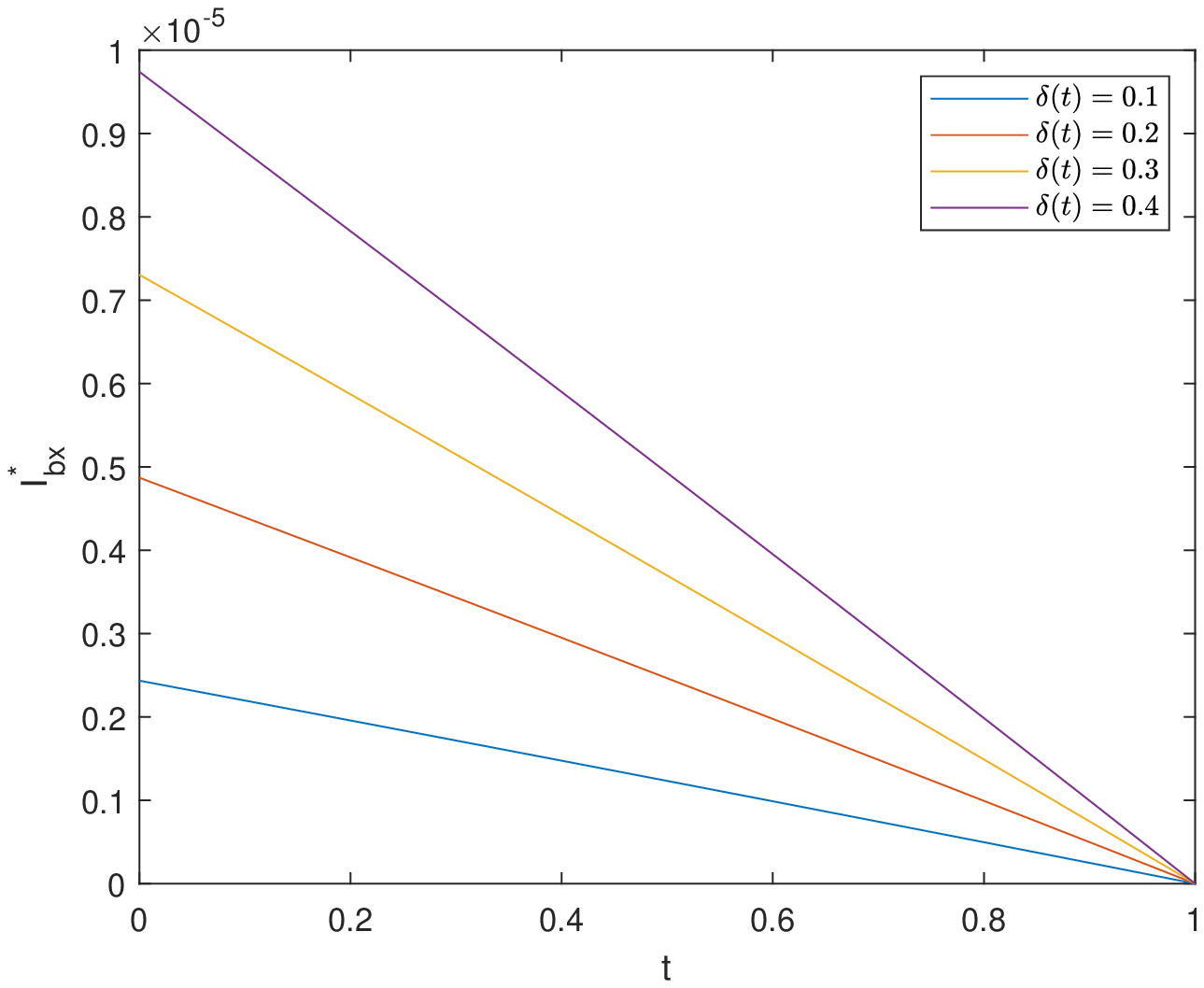}
    \end{minipage}
}
\subfigure[The impact of $\delta(t)$ on $I^*_{b0}$]
{
    \begin{minipage}{7cm}
    \centering
    \includegraphics[scale=0.4]{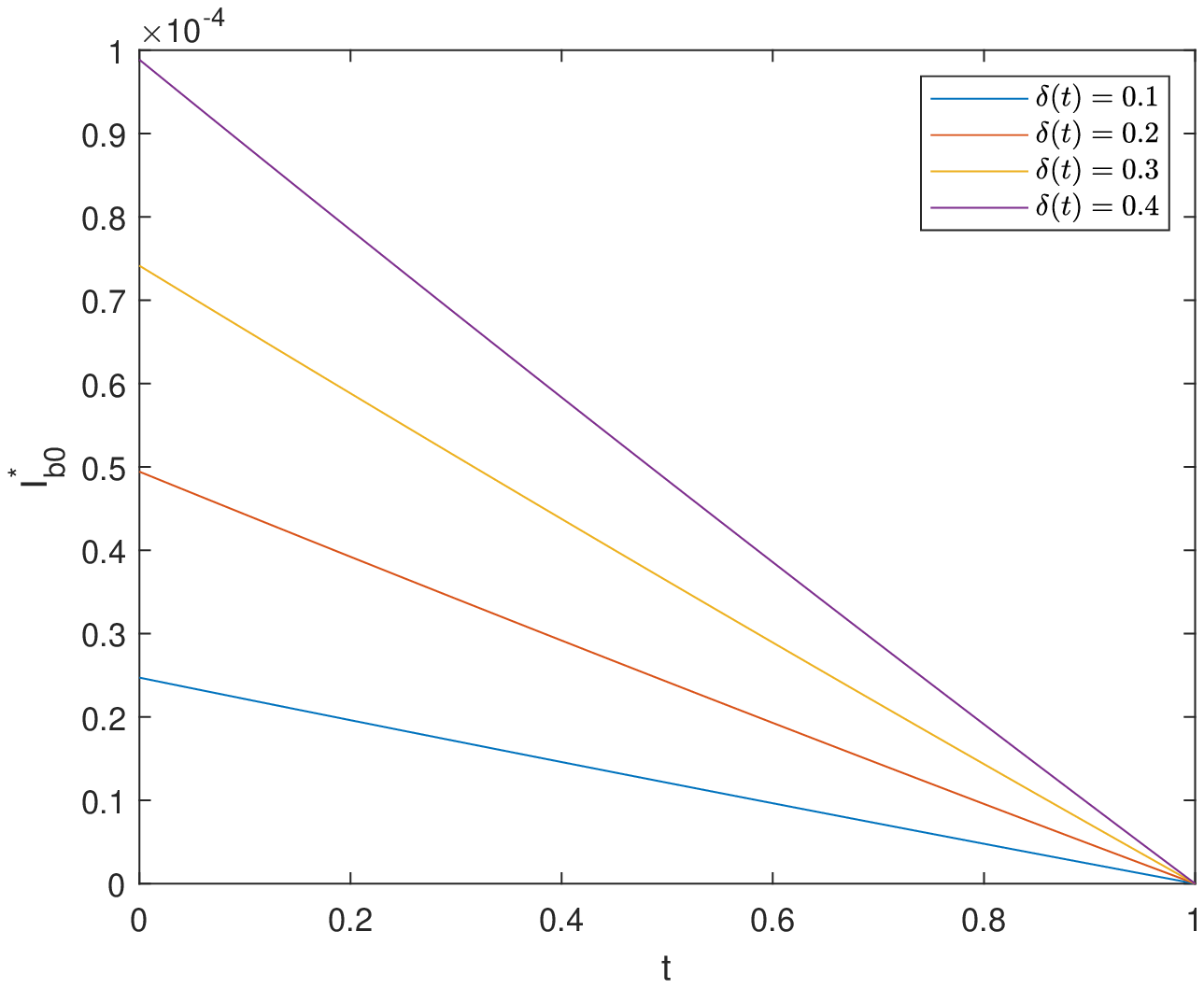}
    \end{minipage}
}
\caption{The impact of $\delta(t)$ on $I^*_b$}
\label{fig:delta-Ib}
\end{figure}

\section{Concluding remarks}

\hspace{0.4cm} In this paper, we investigate a high-dimensional mixed leadership stochastic differential game in a finite time interval under the feedback information mode. And the drift term and diffusion term of the state equation of the game model we consider both contain the control variables of participants, which is absolutely not a trivial extension of Bensoussan et al. \cite{BCCS2019}. Consequently, the system of coupled HJB equations in our verification theorem of the feedback Stackelberg-Nash equilibrium is a system of coupled fully nonlinear parabolic PDEs.\\
\indent Next, we apply the verification theorem gained in section \ref{section 2.2} to deal with the dynamic innovation and pricing decision problem in the supply chain. Unlike Song et al. \cite{SCDL2021}, we study another situation that actually exists in the supply chain, where the buyer serves as the leader of pricing strategies. Therefore, it is necessary to use the mixed leadership game framework to solve the problem of innovation and pricing strategies in this situation. According to the verification theorem, we explicitly express the feedback equilibrium strategies of innovation and pricing with the solutions of coupled Riccati equations $(\ref{riccati P-2})$-$(\ref{riccati N-0})$. For the coupled Riccati equations $(\ref{riccati P-2})$-$(\ref{riccati N-0})$, we obtain the local existence and uniqueness of the solutions based on Picard-Lindel{\"o}f's Theorem. In the end, we numerically analyze the sensitivity of the feedback equilibrium strategies of innovation and pricing with respect to model parameters $C_0$ and $\delta(t)$, and get some managerial insights.


\end{document}